\documentclass[10pt]{amsart}
\usepackage{graphicx}
\usepackage{amsmath,amsthm,amssymb,enumerate}
\usepackage{amscd}
\usepackage{euscript,mathrsfs}
\usepackage[citecolor=blue,colorlinks=true]{hyperref}
\usepackage{mathpazo}
\usepackage[left=3cm,right=3cm,top=4cm,bottom=4cm]{geometry}
\usepackage[shortlabels]{enumitem}
\setenumerate{label={\rm (\alph{*})}}
\usepackage{color}
\catcode`\@=11 \@addtoreset{equation}{section}

\catcode`\@=12

\allowdisplaybreaks

\newtheorem{Theorem}{Theorem}[section]
\newtheorem*{Theorem*}{Theorem}
\newtheorem{Proposition}[Theorem]{Proposition}
\newtheorem{Lemma}[Theorem]{Lemma}
\newtheorem{Corollary}[Theorem]{Corollary}

\theoremstyle{definition}
\newtheorem{Definition}[Theorem]{Definition}
\newtheorem{Remark}[Theorem]{Remark}

\newcommand{\bTheorem}[1]{
\begin{Theorem} \label{T#1} }
\newcommand{\eT}{\end{Theorem}}

\newcommand{\bProposition}[1]{
\begin{Proposition} \label{P#1}}
\newcommand{\eP}{\end{Proposition}}

\newcommand{\bLemma}[1]{
\begin{Lemma} \label{L#1} }
\newcommand{\eL}{\end{Lemma}}

\newcommand{\bCorollary}[1]{
\begin{Corollary} \label{C#1} }
\newcommand{\eC}{\end{Corollary}}

\newcommand{\bRemark}[1]{
\begin{Remark} \label{R#1} }
\newcommand{\eR}{\end{Remark}}

\newcommand{\bDefinition}[1]{
\begin{Definition} \label{D#1} }
\newcommand{\eD}{\end{Definition}}

\DeclareMathOperator{\id}{id}
\newcommand*{\defeq}{\mathrel{\vcenter{\baselineskip0.5ex \lineskiplimit0pt
			\hbox{\scriptsize.}\hbox{\scriptsize.}}}=}
\newcommand{\centeredcirc}[1]{\vcenter{\hbox{$#1\circ$}}}
\newcommand{\smallcirc}{\mathbin{\mathchoice{\centeredcirc\scriptstyle}{\centeredcirc\scriptstyle}{\centeredcirc\scriptscriptstyle}{\centeredcirc\scriptscriptstyle}}}
\newcommand{\T}{\mathbb{T}}
\newcommand{\Z}{\mathbb{Z}}
\newcommand{\N}{\mathbb{N}}
\newcommand{\bH}{\textbf{H}}
\newcommand{\bD}{\textbf{D}}
\newcommand{\bX}{\textbf{X}}
\newcommand{\bT}{\textbf{T}}
\newcommand{\bZ}{\textbf{Z}}
\newcommand{\bC}{\textbf{C}}
\newcommand{\bL}{\textbf{L}}
\newcommand{\bW}{\textbf{W}}
\newcommand{\bS}{\textbf{S}}

\newcommand{\bR}{\mathbf{R}}
\newcommand{\bN}{\mathbf{N}}
\newcommand{\W}{\mathbf{W}}
\newcommand{\clA}{\mathcal{A}}
\newcommand{\clB}{\mathcal{B}}
\newcommand{\clE}{\mathcal{E}}
\newcommand{\clL}{\mathcal{L}}
\newcommand{\clX}{\mathcal{X}}
\newcommand{\clS}{\mathcal{S}}
\newcommand{\clP}{\mathcal{P}}
\newcommand{\clU}{\mathcal{U}}

\newcommand{\les}{\lesssim}

\newcommand{\tmop}[1]{\ensuremath{\operatorname{#1}}}

\DeclareMathOperator*{\argmin}{arg\,min}

\newcommand{\loc}{{\textrm{loc}}}

\allowdisplaybreaks

\newcommand\dela[1]{}

\newcommand{\bFormula}[1]{
\begin{equation} \label{#1}}
\newcommand{\eF}{\end{equation}}

\definecolor{Cgrey}{rgb}{0.85,0.85,0.85}
\definecolor{Cblue}{rgb}{0.50,0.85,0.85}
\definecolor{Cred}{rgb}{1,0,0}
\definecolor{fancy}{rgb}{0.10,0.85,0.10}

\newcommand\Cbox[2]{%
    \newbox\contentbox%
    \newbox\bkgdbox%
    \setbox\contentbox\hbox to \hsize{%
        \vtop{
            \kern\columnsep
            \hbox to \hsize{%
                \kern\columnsep%
                \advance\hsize by -2\columnsep%
                \setlength{\textwidth}{\hsize}%
                \vbox{
                    \parskip=\baselineskip
                    \parindent=0bp
                    #2
                }%
                \kern\columnsep%
            }%
            \kern\columnsep%
        }%
    }%
    \setbox\bkgdbox\vbox{
        \color{#1}
        \hrule width  \wd\contentbox %
               height \ht\contentbox %
               depth  \dp\contentbox
        \color{black}
    }%
    \wd\bkgdbox=0bp%
    \vbox{\hbox to \hsize{\box\bkgdbox\box\contentbox}}%
    \vskip\baselineskip%
}

\date{}

\begin{document}


\title[RDS generated by 3D NSE with transport noise]{Random Dynamical System generated by the 3D Navier-Stokes equation with rough transport noise}

\author{Jorge Cardona}
\address[J. Cardona]{Technical Universit\"at Darmstadt}
\email{jorge@cardona.co}

\author{Martina Hofmanov\'a}
\address[M. Hofmanov\'a]{Fakult\"at f\"ur Mathematik, Universit\"at Bielefeld, D-33501 Bielefeld, Germany}
\email{hofmanova@math.uni-bielefeld.de}

\author{Torstein Nilssen}
\address[T. Nilssen]{Department of Mathematical Sciences, University of Adger, Norway}
\email{torstein.nilssen@uia.no}

\author{Nimit Rana}
\address[N. Rana]{Fakult\"at f\"ur Mathematik, Universit\"at Bielefeld, D-33501 Bielefeld, Germany}
\email{nrana@math.uni-bielefeld.de}

\thanks{
The financial support by the German Science Foundation DFG through  the Research Unit FOR 2402 is greatly acknowledged. This project has received funding from the European Research Council (ERC) under the European Union's Horizon 2020 research and innovation programme (grant agreement No. 949981).  }

\begin{abstract}

We consider the Navier-Stokes system in three  dimensions perturbed by a transport noise which is sufficiently smooth in space and rough in time. The existence of a weak solution was proved in \cite{Hofmanova_etal_2019}, however, as in the deterministic setting the question of uniqueness remains a major open problem. An important feature of systems with uniqueness is the semigroup property satisfied by their solutions. Without uniqueness, this property cannot hold generally.  We select a system of solutions satisfying the semigroup property with appropriately shifted rough path. In addition, the selected solutions respect the well accepted admissibility criterium for physical solutions, namely, maximization of the energy dissipation.  Finally, under suitable assumptions on the driving rough path, we show that the  Navier-Stokes system generates a measurable random dynamical system. To the best of our knowledge, this is the first construction of a measurable single-valued random dynamical system in the state space for an SPDE without uniqueness.

\end{abstract}



\maketitle

\tableofcontents

\section{Introduction}
By now, the question of generating a random dynamical system by an It\^{o}-type stochastic differential equation (SDE) is well settled under natural assumptions on the coefficients implying in particular uniqueness, refer to \cite{Arnold_1998B,Elworthy_1982B} for details. The crucial step consists in proving the stochastic flow property by using the Kolmogorov continuity theorem \cite[Section 1.4]{Kunita_1990B} about the existence of a continuous random field with a finite-dimensional parameters range. Upon establishing the flow property, the existence of a random dynamical system induced by solution to considered SDE follows as derived in \cite{Scheutzow_1996}.

Unfortunately, it is in general not possible to extend the Kolmogorov continuity theorem to an infinite-dimensional parameter range and to obtain stochastic flows for It\^{o}-type stochastic PDEs (SPDEs). The main issue here is that the solution to an It\^{o}-type SPDE is defined almost surely where the exceptional set depends on the initial condition. Nevertheless, it is possible to obtain a stochastic flow, and in particular a random dynamical system, for It\^{o}-type SPDEs driven by special noises provided uniqueness can be shown, refer to \cite{Flandoli_1995B} for details. For instance, if the SPDE is driven by either additive noise or linear multiplicative noise, then it can be  transformed into a random evolution equation. If additionally uniqueness can be proved for the transformed system then it  is possible to construct the associated  stochastic flow  and to generate a random dynamical system. In the context of 2D stochastic Navier-Stokes equations, which has a unique global solution, this approach has been used successfully  by Crauel and Flandoli \cite{Crauel+Flandoli_1994} in  bounded domain and Brze\'zniak and Li  \cite{ZB+Li_2006} in an unbounded domain. 

Due to non-uniqueness of global solutions to stochastic Navier-Stokes equations in three space dimensions, the above method fails. One way to overcome this difficulty, as given in Sell \cite{Sell_1996},  is to replace the state space by the phase space consisting of full trajectories of  solutions of the equation. By modifying the  Sell concept,  Flandoli  and Schmalfuss \cite{Flandoli+Schmalfuss_1996} were able to work with single-valued cocycles and proved the global existence of a random attractor for  3D stochastic Navier-Stokes equation with a multiplicative white noise.  Another way to talk about a random dynamical system generated by an SPDE, which is not known to have a unique global solution, was established in  \cite{Rubio+Robinson_2003}, where the authors introduced a set-valued stochastic semiflow and a set-valued random dynamical system. In particular, by generalizing the idea of Ball \cite{Ball_1997} to the stochastic setting, they found a certain attracting set that depends on the random parameter  and concluded the existence of a random attractor for the 3D Navier-Stokes equations driven by white noise. 

As described above briefly, for some very specific structure of the noise driving an SPDE, a suitable transformation of the equation into a pathwise evolution equation allows to prove the existence of  a random dynamical system. This forces us to rethink and consider a pathwise approach to construct solutions of an SPDE such that they satisfy the cocycle property. Recently, various techniques have been derived  to give a pathwise meaning to the solutions of SPDEs by exploiting the ideas of Lyons' rough paths theory \cite{Lyons_1998} in one way or other. 
However, there are very few results that explore the pathwise properties of the solutions to study the long time behavior of stochastic flows.  For example, in \cite{Schmalfuss_etal_2010,Schmalfuss_etal_2016} the authors deal  with random dynamical systems for SPDEs driven by a fractional Brownian motion with Hurst parameter $H \in (1/2, 1)$ and $H \in (1/3, 1/2]$. The proof relies on the definition of a pathwise stochastic integral based on the Riemann-Liouville fractional derivatives. Using regularizing properties of analytic semigroups, the authors in \cite{Hesse+Neamtu_2020} proved the existence of a unique global  solution to infinite-dimensional parabolic rough evolution equations and investigated random dynamical systems for such equations.  Existence of a unique random attractor applicable for a large class of SPDEs perturbed by Brownian motion, fractional Brownian motion and L\'evy-type noise is shown in \cite{Gess_etal_2011}. The results obtained in \cite{Gess_etal_2011} are based on the variational approach to SPDE. Recently, the authors in \cite{Hofmanova_etal_2021+}, defined an intrinsic notion of solution based on ideas from the rough path theory and studied the well-posedness of 2D Navier-Stokes equations in an equivalent vorticity formulation. They derive rough path continuity of the equation and show that for a large class of driving signals, the system generates a continuous random dynamical system.

To the best of our knowledge, all the rough PDEs (RPDEs) considered so far to investigate the induced random dynamical system proved to have a unique global solution. However, from the point of view of applications there is a huge number of physically relevant systems where uniqueness is either unknown or not even valid. Such examples appear for instance in  fluid dynamics and are represented by the iconic example of the Navier-Stokes system. Nevertheless, despite the theoretical difficulties, the Navier-Stokes system and related fluid dynamics equations are widely used in practice and count as a reliable basis for modeling
and simulations. From the mathematics point of view, it is therefore essential to study such models, which are not known to have unique global solutions, and to develop methods to investigate their long time behavior.

Taking a step in this direction,   we study the existence of a random dynamical system  induced by   solutions of a  system of Navier-Stokes equations on the three-dimensional torus $\bT^3$ subject to a rough transport noise. The noise arises from perturbing the transport advecting velocity field by a space-time dependent noise and is, at least formally, energy conservative.
The rough path philosophy of splitting analysis from probability, as well as a Wong-Zakai stability result are the key ingredients of our construction. For the first time, 
we are able to construct a measurable single-valued random dynamical system in the state space for an SPDE without uniqueness. Our proof relies on a selection procedure and is direct, benefiting from the rough path theory rather than a transformation into a random PDE.

The system of interest governs the  evolution of the velocity field $u: \bR_+ \times \bT^3 \rightarrow \bR^3$ and the pressure $p : \bR_+ \times \bT^3 \rightarrow \bR$ of an incompressible viscous fluid and reads as
\begin{equation} \label{eq:classicalForm}
\begin{aligned}
\partial_t u +(u-\dot{a})\cdot \nabla u +\nabla p& =   \Delta u,  \\
\nabla \cdot u & = 0,  \\
u(0) & = u_0 \in L^2(\bT^3;\bR^3).
\end{aligned}
\end{equation}
Here $\dot{a}$ is the (formal) derivative in time of a function $a = a_t (x) : \bR_+ \times \bT^3 \to \bR^3$ which can have the following factorization:
\begin{equation}\label{eq:aa3}
a_t(x) = \sigma_k(x) z^k_t=  \sum_{k=1}^{K} \sigma_k(x) z^k_t,
\end{equation}
where we adopt the convention of summation over repeated indices $k\in \{1,\ldots,K\}$. We also assume that for all $k\in \{1,\ldots,K\}$, the vector fields $\sigma_k:\bT^3 \rightarrow \bR^3$ are bounded, divergence-free, and twice-differentiable with uniformly bounded first and second derivatives. The driving signal $z$ is assumed to be a  $\bR^{K}$-valued $\alpha$-H\"older path for some $\alpha \in \left( \frac{1}{3}, \frac{1}{2}\right]$ which can be lifted to a geometric rough path $\bZ=(Z,\mathbb{Z})$, see Section \ref{sec:weakSoln} for precise assumptions.

The system \eqref{eq:classicalForm}, \eqref{eq:aa3} was studied recently in \cite{Hofmanova_etal_2019} within the framework  of unbounded rough drivers as introduced in \cite{Bailluel+Gub_2017} and further developed in \cite{Deya_etal_2019} and other works, e.g. \cite{Flandoli_etal_2020}. The authors in \cite{Hofmanova_etal_2019} introduced a suitable notion of weak solution to \eqref{eq:classicalForm} and proved existence, see \cite[Theorem 2.13]{Hofmanova_etal_2019}. The proof  relies on a Galerkin approximation in combination with uniform energy estimates of the solution as well as the corresponding remainder terms and a compactness argument. However, it turns out that this notion of solution is not suitable for the construction of the random dynamical system. More precisely, as the kinetic energy of weak solutions is only lower semicontinuous, the so-called energy sinks appear and cannot be avoided. Consequently, at the time of each energy sink, the actual energy of the system is not controlled by the energy of the solution. This is a problem since in a random dynamical system  the future evolution after each time $t$ has to be fully determined by the value of the solution at $t$.

Inspired by \cite{Basaric_2020, Breit_etal_2020,Breit_etal_2020-a}, we overcome this issue by including an auxiliary variable, i.e. the energy $E$, as a part of the solution. Accordingly,  a weak solution in our context actually consists of a triple $[u,p,E]$ of the fluid velocity and the pressure together with the energy. In addition, the Navier-Stokes system \eqref{eq:classicalForm}, \eqref{eq:aa3} needs to be  supplemented by a suitable form of an energy inequality. As usual, see for e.g. \cite[Section 4.1.2]{Hofmanova_etal_2019}, the pressure can be recovered from the velocity so  we do not specify it in our results or in the definition of a solution, see Definition~\ref{def-weaksolution}. Including the energy as an additional datum may seem a bit superfluous at first sight, but  it is indispensable in order to obtain the semigroup property for every time. Otherwise one would only obtain a statement for a.e. time as in \cite{Basaric_2020, Flandoli+Romito_2008}. We refer to Remark~\ref{r3.10} and to \cite[Section 6]{HZZ20} for a further discussion of this issue.

The first main result of the current paper can be stated as follows, see Theorem~\ref{thm-semiflowSel} for the precise formulation.

\begin{Theorem*}  
	The Navier-Stokes equation \eqref{eq:classicalForm}-\eqref{eq:aa3}
	 admits a semiflow selection  in the class of weak solutions, that is, there is a measurable mapping
	\[
	U: [u_0,E_0,\bZ] \mapsto [u,E],
	\]
	which assigns to every initial condition $[u_{0},E_{0}]$ and a rough path $\bZ$ one solution trajectory $[u,E]$ so that 
 the following \textbf{semigroup property} holds true
	\[
	U \{ u_0, E_0,\bZ\}(t_{1}+t_{2}) =
	U \{ U\{t_{1}, u_{0},E_{0},\bZ \}, \tilde{\bZ}_{t_1}\}(t_{2}), \quad \textrm{ for any } t_1, t_2\geq 0,
	\]
	where $\tilde{\bZ}_{t_1}(\cdot) $ denotes the shifted rough path $\tilde{\bZ}_{t_1}(\cdot) := \bZ(t_1+\cdot)$\footnote{With a slight abuse of notation we write $\bZ(t_1+\cdot)$ for the 2-index map $(Z,\mathbb{Z})_{t_{1}+\cdot,t_{1}+\cdot}$.}.
\end{Theorem*}

To prove this result we adapt to the rough path setting the selection procedure introduced in \cite{Cardona+Kapitanski_2020} and further generalized in \cite{Basaric_2020, Breit_etal_2020,Breit_etal_2020-a}. The key property which then permits to deduce the existence of a measurable random dynamical system is the rough path stability in the spirit of a Wong-Zakai approximation result, see Theorem \ref{thm-SeqStabilitywrtZandD}. Indeed, this permits to go back to probability and consider random driving rough paths which satisfy a suitable cocycle property as introduced in \cite{BRS_2017}. An example is a fractional Brownian motion with Hurst parameter $H>\frac13$.
This brings us to the second main result of this paper. See Theorem~\ref{thm-rds} for details and a precise formulation.
\begin{Theorem*}
	Under the assumption that the driving path $\bZ$ is an  $\alpha$-H\"older rough path cocycle, see \cite{BRS_2017}, the Borel measurable map 
	$$\Phi: (t,\omega,[u_0,E_0]) \mapsto U\{ u_0,E_0, \bZ(\omega)\},$$
	is a measurable random dynamical system over a measurable metric dynamical system $(\Omega,\mathcal{F},\mathbb{P}, (\theta_t)_{t \in \T})$, that is, it satisfies the \textit{cocycle property}, which roughly states that 
		\begin{equation*}
		\Phi(t+s, \omega) = \Phi(t, \theta_s\omega) \smallcirc \Phi(s, \omega) \quad \forall s,t \in \mathbb{T}, \omega \in \Omega. 
		\end{equation*}
\end{Theorem*} 

We note that the possibility of changing measurability of the random dynamical system  above into its continuity remains a major open question.  Indeed, this corresponds to continuity with respect to the initial condition, which in the deterministic setting has not been proved so far. In the stochastic setting, on the other hand, the probabilistic counterpart, i.e. the Feller property, was established only under the assumption that the noise is additive and sufficiently non-degenerate, see \cite{DPD,Flandoli+Romito_2008}.

The organization of the  paper is as follows. In Section \ref{sec:notation}, we introduce our notation and provide the required definitions with the notion of weak solution used in the paper.  In Section~\ref{sec:semiflow},  we introduce the concept of a semiflow selection in terms of the two state variables: the velocity and the energy. We also analyze the properties (compactness, shift invariance and
 continuation) of the solution set and prove the  existence of a semiflow selection, refer to Theorem~\ref{thm-semiflowSel}. Section \ref{sec:rds} is devoted to the existence of a random dynamical system, Theorem \ref{thm-rds}, which is a central result of
this paper. We conclude the paper with Appendix \ref{appx:apriori_compactness} where we state  all the required a priori estimates from  \cite[Section 3]{Hofmanova_etal_2019} and state a compact embedding lemma which is helpful in  the proof of the crucial sequential stability result, see Theorem \ref{thm-SeqStabilitywrtZandD}.

\section{Notation}\label{sec:notation}

In this section we fix the notation   which is used throughout the paper. Since we intend to construct  a semiflow for weak solutions to \eqref{eq:classicalForm} whose existence is proven in \cite{Hofmanova_etal_2019}, to avoid any notational confusion, we closely follow \cite[Section 2]{Hofmanova_etal_2019}  and only state the required definitions and their properties.

We  write $a \les b$  if  there exists a positive constant $C$ such that $a \le C b$. If the constant $C$ depends only on the parameters $p_1,\ldots,p_n$, we  write $C(p_1, \ldots, p_n)$ and $\les_{p_1,\ldots,p_n}$, respectively, instead of $C$ and $\les$.
Let $\T:= [0,\infty)$. We let  $\bN$ to be  the set of natural numbers.  Let  $\bT^3$ denotes the three dimensional flat torus.

For given Banach spaces $V_1$ and $V_2$, the space of continuous linear operators from $V_1$ to $V_2$ will be denoted by $\clL(V_1,V_2)$. Note that $\clL(V_1,V_2)$ is endowed with the operator norm which we denote by $|\cdot|_{\clL(V_1,V_2)}$.
For a given $\sigma$-finite measure space $(X,\clX,\mu)$, separable Banach space $V$ with norm $|\cdot|_V$, and $p\in [1,\infty]$,  we denote by $L^p(X;V)$ the Bochner space of  strongly-measurable and $L^{p}$-integrable functions $f: X\rightarrow V$.
For a given Hilbert space $H$ and $T >0$, we let
$
L^2_TH=L^2([0,T];H)$ and $   L^{\infty}_TH=L^\infty([0,T];H).
$
Moreover, we set $\bL^2=L^2( \bT^3 ; \bR^3)$.
We use subscript $``\loc"$ to point out that the elements restricted to any bounded interval $J$ belong to the space with domain $J$. For instance, by $L_{\loc}^\infty(\T;\bR)$ we understand the set of  all $\mu$-equivalence-classes of strongly-measurable functions $f: \T \to \bR$ such that $f|_{J} \in L^\infty(J;\bR)$ for every bounded $J \subset \T$. 
We write  $C_TH=C([0,T];H)$ to denote  the Banach space  of continuous functions from $[0,T]$ to $H$, endowed with the supremum norm in time. Moreover, if $H$ is subject to weak topology, then we write  $C([0,T];H_w)$.

Let $\clS$ be the Fr\'echet space of infinitely differentiable periodic complex-valued functions with the usual set of seminorms.
Let $\clS'$ be the continuous dual space of $\clS$ endowed with the weak-star topology.  In extension of these, we write $\bS'=(\clS')^d$ for the set of continuous linear functions from $\bS=(\clS)^d$ to $\bC$ endowed with the weak-star topology.

For a given $\beta\in \bR$, the Hilbert space $\W^{\beta,2}$  is defined as
$$
\W^{\beta,2} :=(I-\Delta)^{-\frac{\beta}{2}}\bL^2=\{f\in \bS': (I-\Delta)^{\frac{\beta}{2}}f\in \bL^2\},
$$
with inner product
$$
(f,g)_{\beta}:=((I-\Delta)^{\frac{\beta}{2}}f,(I-\Delta)^{\frac{\beta}{2}}g)_{\bL^2}  ,\quad  \; f,g\in \W^{\beta,2},
$$
and  induced norm $|\cdot |_{\beta}$. For notational simplicity, when $\beta=0$ we omit the index in the inner product, i.e. $(\cdot, \cdot) := (\cdot, \cdot)_0$.   Let
$$
\bH^0 : =\left\{f\in \W^{0,2}: \;\nabla \cdot  f= 0\right\}.
$$

\subsection{Helmholtz--Leray projection}\label{sec:HelmholtzLeray_projection}
We denote the Helmholtz--Leray projection by $P :\bS'\rightarrow \bS'$, which is well-known in the study of Navier-Stokes equation, see \cite{RT83}, and let  $Q = I - P$. It follows  that  $P,Q\in \clL(\W^{\beta,2},\W^{\beta,2})$ and that the operator norms of $P$ and $Q$  are bounded by one for all $\beta\in \bR$.

By setting
$$\bH^{\beta} : =P\W^{\beta,2} \qquad \textrm{ and } \qquad \bH_{\perp}^{\beta}: =Q\W^{\beta,2},$$  
it can be showed that  for all $\beta\in \bR$, see for e.g. \cite[Lemma 3.7]{mikulevicius2003cauchy},
$$
\W^{\beta,2} =\bH^\beta\oplus \bH_{\perp}^\beta,
$$
where
$$
\bH^\beta=\left\{f\in \W^{\beta,2}: \;\nabla \cdot  f= 0\right\},
$$
$$
\bH_{\perp}^\beta=\{g\in \W^{\beta,2}:  \langle f,g\rangle_{-\beta,\beta}=0, \;\; \forall f\in \bH^{-\beta} \}.
$$

The following part of this subsection will shed light on the construction of the unbounded rough drivers associated with \eqref{eq:classicalForm}.
Let $\sigma:\bT^3\rightarrow \bR^3$ be  twice differentiable and divergence-free. Moreover, assume that the derivatives of $\sigma$ up to order two are bounded uniformly by a constant $M_0$.

Let  $\clA^1: =\sigma\cdot \nabla$  and  $\clA^2 :=(\sigma\cdot \nabla)(\sigma\cdot \nabla).$ It follows  that there is a  constant  $M$ (depends on $M_0,\beta$)  such that 
$$
|\clA^1|_{\clL(\W^{\beta+1,2},\W^{\beta,2})}\le M, \;\forall \beta\in [0,2],\qquad
|\clA^2 f|_{\clL(\W^{\beta+2,2},\W^{\beta,2})}\le M, \;\forall \beta\in [0,1].
$$
We ask the reader to see \cite{mikulevicius2001note} for such estimates on the whole space, but they can easily be adapted to the periodic setting as required in the current paper.

Since  $P\in \clL(\W^{\beta,2},\bH^\beta)$ and $Q\in \clL(\W^{\beta,2}, \bH^\beta_{\perp})$ for all $\beta\in \bR$,  both of which have operator norm bounded by $1$, we have 
\begin{align}
& |P \clA^1|_{\clL(\bH^{\beta+1},\bH^\beta)}\le M,  \forall \beta\in [0,2], \label{A12bound-1}\\  
& |P \clA^2|_{\clL(\bH^{\beta+2},\bH^\beta)}\le M, \forall \beta\in [0,1], \label{A12bound-2}
\end{align}
and hence $(P \clA^1)^*\in \clL((\bH^{\beta})^*, (\bH^{\beta+1})^*)$ for $\beta\in [0,2]$ and  $(P \clA^2)^*\in \clL((\bH^{\beta})^*, (\bH^{\beta+2})^*)$  for $\beta\in [0,1]$.

To analyze the convective term, we employ the following notation and bounds. Owing to  \cite[Lemma 2.1]{RT83} adapted to fractional norms, the trilinear form defined by
$$
b(u,v,w) : = \int_{\bT^3} ((u\cdot \nabla)v)\cdot w \, \, dx=\sum_{i,j=1}^3\int_{\bT^3} u^i D_iv^j w^j \,\, dx,
$$
satisfies 
\begin{equation}\label{trilinear form estimate}
b(u,v,w)\les_{\beta_1,\beta_2,\beta_3} |u|_{\beta_1}|v|_{\beta_2+1}|w|_{\beta_3}.
\end{equation}
for every $\beta_1,\beta_2,\beta_3\in \bR_+$   such that \begin{gather*}
\beta_1 + \beta_2 + \beta_3 \geq \frac{3}{2}, \;\; \textnormal{ if }\;\; \beta_i \neq \frac{3}{2} \;\textnormal{ for all } i\in \{1,2,3\}, \\ \beta_1 + \beta_2 + \beta_3 > \frac{3}{2}, \;\; \textnormal{ if } \beta_i = \frac{3}{2}\; \textnormal{ for some }  i\in \{1,2,3\}. 
\end{gather*}

Furthermore, for all $u\in  \bH^{\beta_1}$ and  $(v,w)\in \bW^{\beta_2+1,2}\times \bW^{\beta_3,2}$ such that $\beta_1,\beta_2,\beta_3$ satisfy \eqref{trilinear form estimate}, we have
\begin{equation}\label{eq:B prop}
b(u,v,w)=-b(u,w,v) \quad \textnormal{and} \quad b(u,v,v)=0.
\end{equation}
For $\beta_1,\beta_2,$ and $\beta_3$ that satisfy \eqref{trilinear form estimate} and any given $(u,v)\in \bW^{\beta_1,2}\times \bW^{\beta_2+1,2}$, we define $B(u,v)\in \bW^{-\beta_3,2}$ by
$$
\langle B(u,v),w\rangle_{-\beta_3,\beta_3}=b(u,v,w), \quad \forall w\in \bW^{\beta_3,2}.
$$
In last, we define $B_P=PB$ and note that $$B_P:=PB: \bW^{\beta_1,2}\times \bW^{\beta_2+1,2}\rightarrow \bH^{-\beta_3}, $$
for $\beta_1,\beta_2,$ and $\beta_3$ that satisfy \eqref{trilinear form estimate}.
We set $$B(u)=B(u,u), \quad \textrm{ and } \quad  B_P(u) := B_P(u,u).$$

\subsection{Smoothing operators}\label{sec:SmoothingOper}

Motivated by \cite{Bailluel+Gub_2017}, we also need a family of self-adjoint smoothing operators $(J^{\eta})_{\eta \in (0,1]}$ such that for all $\beta \in  \mathbf{R}$ and $\gamma \in \bR_+$,
\begin{equation} \label{smoothingOperator}
|(I - J^{\eta}) f |_{\beta} \les \eta^{\gamma} |f|_{\beta+ \gamma} \hspace{.5cm} \textnormal{ and} \hspace{.5cm} | J^{\eta} f |_{\beta+\gamma} \les \eta^{-\gamma} |f|_{\beta}.
\end{equation}
For a construction of one such family, we refer \cite[Section 2.2]{Hofmanova_etal_2019}.

\subsection{Rough path theory} \label{ss:rp}
For a given interval $I \subset \bR,$ we set 
\begin{align*}
\Delta_I   := \{(s,t)\in I^2: s\le t\},  \qquad  \Delta^{(2)}_I  := \{(s,\theta,t)\in I^3: s\le\theta\le t\}. 
\end{align*}
For a path $f: I \to \bR^K$ we define its increment as $ \delta f_{st} := f_t-f_s, \forall s,t \in I$ and for a two-index map $g: \Delta_{I}\rightarrow \bR$, we define the second order increment operator $$\delta g_{s \theta t} := g_{st} - g_{\theta t} - g_{s \theta}, \quad \forall (s,\theta,t)\in\Delta^{(2)}_I.$$

Let $\alpha >0$ and $J$ be a bounded interval in $\bR$. We denote by $C_2^{\alpha}(J; \bR^K)$ the closure of the set of smooth $2$-index maps $g: \Delta_J \to \bR^K$ with respect to the H\"older coefficient  
\begin{align*}
[g]_{\alpha,J}  := 	\sup_{s,t \in \Delta_J, s\neq t}  \frac{|g_{st}|}{|t-s|^\alpha} < \infty. 
\end{align*}
Note that, since the zero element is $g_{st} =0$ for all $(s,t) \in \Delta_J$,  we infer that $[g]_{\alpha,J} $ is actually a norm on $C_2^{\alpha}(J; \bR^K)$. Note that with this definition, the space $C_2^{\alpha}(J; \bR^K)$   is  Polish. By $C^{\alpha}(J; \bR^K)$ we denote the closure  of the set of smooth paths $f : J \to \bR^K$ w.r.t. the semi-norm $[\delta f]_{\alpha,J}$. 
By $C_{2,\loc}^{\alpha}(\T; \bR^K)$  we denote the space of $2$-index maps $g : \Delta_{\T} \to \bR^K$ such that the restriction of $g$ on every bounded interval $J \subset \T$, which we denote by $g|_{\Delta_J} $, belongs to $C_{2}^{\alpha}(J; \bR^K)$.

Next, we present the definition of an $\alpha$-H\"older rough path. A detailed exposition of rough path theory can be found in \cite{FH_2014}. 
\begin{Definition}\label{defi-rough-path}
	Let $K \in \bN$ and $\alpha \in \left(\frac{1}{3}, \frac{1}{2} \right]$. An $\alpha$-H\"older rough path is a pair 
	\begin{equation}\label{alpha-rpLoc}
	\bZ =(Z , \Z) \in C_{2,\loc}^{\alpha}(\T; \bR^K) \times C_{2,\loc}^{2\alpha}(\T; \bR^{K\times K}),
	\end{equation}
	such that for every bounded interval $J$ in $\T$
	\begin{equation}\label{alpha-rp}
	(Z|_{\Delta_J} , \Z|_{\Delta_J}) \in C_2^{\alpha}(J; \bR^K) \times C_2^{2\alpha}(J; \bR^{K\times K}),
	\end{equation}
	and satisfies the Chen's relation
	\begin{equation*}\label{chen-rela}
	\delta \mathbb{Z}_{s\theta t}=Z_{s\theta} \otimes Z_{\theta t},  \quad \forall(s, \theta, t) \in \Delta^{(2)}_{J}.
	\end{equation*}
	An $\alpha$-H\"older rough path $\mathbf{Z}=(Z, \mathbb{Z})$ is said to be geometric if the restriction $\bZ|_{J}$ can be obtained as the limit in the   product topology   of a sequence of rough paths  $\{(Z^{n},\mathbb{Z}^{n})\}_{n \in \bN}  \subset C_2^{\alpha}(J; \bR^K) \times C_2^{2\alpha}(J; \bR^{K\times K})$ such that for each $n\in \bN$, 
	$$Z^{n}_{st}:= \delta z^{n}_{st} \quad \textnormal{ and } \quad \mathbb{Z}^{n}_{st}:=\int_s^t \delta z^{n}_{s\theta} \otimes \mathrm{d} z^{n}_\theta ,$$
	for some smooth path $z^n:J  \to \bR^K$, where the  iterated integral is understood in the Riemann sense. 
	
	We denote by $\mathcal{C}^{\alpha}_g(\T;\bR^K)$ the set of all geometric $\alpha$-H\"older rough paths and endow it with the product topology. 
\end{Definition}

For given bounded interval $J \subset \bR$ and $(Z,\Z) \in C^{\alpha}_2 (J;\bR^{K}) \times C^{2\alpha}_2 (J; \bR^{K\times K})$, let us further set 
\begin{equation*}
||Z||_{\alpha,J} :=  \sup_{s,t \in \Delta_J, s\neq t}  \frac{|Z_{st}|}{|t-s|^\alpha}, \qquad ||\Z||_{2\alpha,J} :=  \sup_{s,t \in \Delta_J, s\neq t}  \frac{|\mathbb{Z}_{st}|}{|t-s|^{2\alpha}}, 
\end{equation*}
and 
\begin{equation*}
|||\bZ|||_{\alpha,J}  := ||Z||_{\alpha,J} + ||\Z||_{2\alpha,J}. 
\end{equation*}
Note that $|||\cdot|||_{\alpha,J}$ is a norm on $C^{\alpha}_2 (J;\bR^{K}) \times C^{2\alpha}_2 (J; \bR^{K\times K})$.

We also need to deal with finite $p$-variation spaces. To introduce them let $\mathcal{P}(J)$ denote the set of all partitions of a bounded interval $J$ and let $V$ be a Banach space with norm $| \cdot |_V$. 
A function $g: \Delta_J \to V$ is said to have finite $p$-variation for some $p>0$ on $J$ if 
$$
|g|_{p-\textnormal{var};J;V}:=\sup_{(t_i)\in \clP(J)}\left(\sum_{i}|g _{t_i t_{i+1}}|^p_V \right)^{\frac{1}{p}}<\infty, 
$$
and we denote by $C_2^{p-\textnormal{var}}(J; V)$ the set of all continuous functions with finite $p$-variation on $J$ equipped with the seminorm $|\cdot |_{p-\textnormal{var}; J;V}$.  We denote by $C^{p-\textnormal{var}}(J; V)$ the set of all paths $z : J \rightarrow V$ such that $\delta z \in C_2^{p-\textnormal{var}}(J; V)$.

A two-index map $\omega: \Delta_J \rightarrow [0,\infty)$  is called a control if  \begin{itemize}
	\item it is continuous on $\Delta_J$;
	\item it attains zero on diagonal i.e.,  for all $s\in J$, $\omega(s,s)=0$;
	\item it is superadditive i.e.,  for all $(s,\theta,t)\in \Delta^{(2)}_J$, $ \omega(s,\theta)+\omega(\theta ,t)\le \omega(s,t). $
\end{itemize} 

If for a given $p > 0$, $g \in C^{p-\textnormal{var}}_2(J; V)$, then it is well-known that the 2-index map $\omega_g: \Delta_J \rightarrow [0,\infty)$ defined by
$$
\omega_g(s,t) := |g|_{p-\textnormal{var};[s,t]}^p 
$$
is a control, see for e.g.,  \cite[Proposition 5.8]{FV_2010}. Moreover, in such situation,   $ |g_{st}|_V \leq \omega_g(s,t)^{\frac{1}{p}}$ for all $(s,t)\in \Delta_J$. 

One can equivalently define a semi-norm on $C_2^{p-\textnormal{var}}(J;V)$ as, see \cite[Section 2.3]{Hofmanova_etal_2019},  
\begin{equation}\label{equivdefpvar}
|g |_{p-var; [s,t]} = \inf \{ \omega(s,t)^{\frac{1}{p}} : |g_{uv}|_V \leq \omega(u,v)^{\frac{1}{p}} \textnormal{ for all } (u, v)  \in \Delta_{[s,t]} \} . 
\end{equation}

Motivated by \eqref{equivdefpvar},  we define a local version of the $p$-variation spaces.

\begin{Definition}\label{def:variationSpace}
	Given an interval $J=[a,b]$ for some $a,b \in \T$, a control $\varpi$ on $\Delta_J$, and a positive real number $L$, we denote by $C^{p-\textnormal{var}}_{2, \varpi, L}(J; V)$  the space of continuous two-index maps $g : \Delta_J \rightarrow V$ for which  there exists at least one control $\omega$ such that $|g_{st}|_V \leq \omega(s,t)^{\frac{1}{p}}$ for every $(s,t)\in \Delta_J$ which gives $\varpi(s,t) \leq L$ .
	
	We define a semi-norm on this space by
	$$
	|g |_{p-\textnormal{var}, \varpi,L; J} :=\inf \left\{\omega(a,b)^{\frac{1}{p}} :  \omega  \textnormal{ is a control s.t. } |g_{st}|_V \leq \omega(s,t)^{\frac{1}{p}}, \;\forall (s, t)  \in \Delta_{J} \textnormal{  with } \varpi(s,t) \leq L \right\} . 
	$$ 
	By $C^{p-\textnormal{var}}_{2, \varpi, L,\loc}(\T; V)$ we mean the set of continuous two-index maps $g : \Delta_{\T} \rightarrow V$ such that for every bounded interval $J \subset \T$ the restriction $g|_{\Delta_J}$ belongs to $C^{p-\textnormal{var}}_{2, \varpi, L}(J; V)$. 
\end{Definition}

\medskip

Observe that, since the rough perturbation  in \eqref{eq:classicalForm} is  (unbounded) operator valued, it is necessary to use the notion of unbounded rough drivers, which can be seen as operator valued rough paths with values in a suitable space of unbounded operators, see \cite{Bailluel+Gub_2017}.
In what follows, by scale we mean a family $(E^{\beta}, | \cdot |_{\beta})_{ \beta \in \bR_+}$ of Banach spaces such that $E^{\gamma+ \beta}$ is continuously embedded into $E^{\beta}$ for $\gamma \in \bR_+$. For $\beta \in\bR_+,$ we denote by $E^{-\beta}$ the topological dual of $E^{\beta}$, and note that, in general, $E^{-0}\neq E^0.$

\begin{Definition}
	\label{def:urd}
	Let $\alpha \in \left( \frac{1}{3}, \frac{1}{2}\right]$ and a bounded interval $J\subset \T$ be given. A continuous unbounded $\alpha$-rough driver with respect to the scale $(E^{\beta}, |\cdot |_{\beta})_{\beta \in\bR_+}$, is a pair $\mathbf{A} = (A^1,A^2)$ of $2$-index maps such that
	there exists a  control $\omega_A$ on $J$ such that for every $(s,t)\in \Delta_J$,
	\begin{equation}\label{ineq:UBRcontrolestimates}
	| A^1_{st}|_{\mathcal{L}(E^{-\beta},E^{-(\beta+1)})} \leq (\omega_{A}(s,t))^\alpha \ \  \text{for}\ \ \beta \in [0,2], \quad
	|A^2_{st}|_{\mathcal{L}(E^{-\beta},E^{-(\beta+2)})}  \leq (\omega_{A}(s,t))^{2\alpha} \ \ \text{for}\ \ \beta \in [0,1],
	\end{equation}
	and  the Chen relation holds true, that is,
	\begin{equation}\label{eq:chen-relation}
	\delta A^1_{s\theta t}=0,\qquad\delta A^2_{s\theta t}= A^1_{\theta t}A^1_{s\theta},\;\;\forall (s,\theta,t)\in\Delta^{(2)}_J.
	\end{equation}
\end{Definition}

\subsection{Definition of weak solution}\label{sec:weakSoln}

In this section, we define a notion of a weak solution to  \eqref{eq:classicalForm} and \eqref{eq:aa3}. 

Let $z\in C_{\loc}^{\alpha}(\T;\bR^K)$ be such that it can be lifted to a  continuous geometric  $\alpha$-H\"older rough path $\bZ=(Z,\mathbb{Z})\in \mathcal{C}^{\alpha}_{g,\loc}(\T;\bR^K)$ for some $\alpha\in \left( \frac{1}{3}, \frac{1}{2} \right]$. For each $k\in \{1,\ldots,K\}$, assume that  $\sigma_k : \bT^3 \rightarrow \bR^3$ is  twice differentiable and divergence-free. Moreover, assume that for all $k\in \{1,\ldots,K\}$,  $\sigma_k$ and its derivatives up to order two are bounded uniformly.

Applying the Leray projection $ P :\W^{\alpha,2} \rightarrow \bH^\alpha$ and gradient projection $Q: \W^{\alpha,2} \rightarrow \bH^{\alpha}_{\perp}$, defined in Section \ref{sec:HelmholtzLeray_projection}, separately to \eqref{eq:classicalForm} with \eqref{eq:aa3} yields
\begin{align}
	& \partial_t u + P[ (u\cdot \nabla) u] =  \Delta u+P[ (\sigma_k\cdot \nabla) u] \dot{z}_t^k,  \label{NSDiffFormSystem_u} \\
	& \nabla p+Q[ (u \cdot \nabla)  u] =Q [(\sigma_k \cdot \nabla)  u] \dot{z}_t^k.  \label{NSDiffFormSystem_pi}
\end{align}
By setting 
$$\pi: = \int_0^{\cdot} \nabla p_r\,dr, $$ 
and integrating the  system \eqref{NSDiffFormSystem_u}-\eqref{NSDiffFormSystem_pi} over $[s,t]$ and then iterating the equation into itself we obtain, see \cite[Section 2.5]{Hofmanova_etal_2019} for a complete derivation, 
\begin{align}
& \delta u_{st} +\int_s^t P[ (u_r \cdot \nabla)  u_r]\,dr =\int_s^t \Delta u_rdr+ [A_{st}^{P,1}+A_{st}^{P,2}]u_s +u_{st}^{P, \natural} , \label{NSRoughFormSystem_u}\\
& \delta \pi_{st}+\int_s^t Q[ (u_r \cdot \nabla ) u_r)]\,dr = [A_{st}^{Q,1}+A_{st}^{Q,2}] u_s + u_{st}^{Q, \natural}  , \label{NSRoughFormSystem_pi}
\end{align}
where
\begin{gather*}
A^{P,1}_{st}\varphi :=   P  [(\sigma_k \cdot \nabla)  \varphi]   \, Z_{st}^k, \qquad
A^{P,2}_{st}\varphi:=P[(\sigma_k\cdot\nabla) P [(\sigma_i\cdot\nabla)\varphi]]\mathbb{Z}^{i,k}_{st},\\
A^{Q,1}_{st}\varphi :=    Q  [(\sigma_k \cdot \nabla)  \varphi] \, Z_{st}^{k}, \qquad
A^{Q,2}_{st}\varphi :=   Q[(\sigma_k\cdot\nabla) P [(\sigma_i\cdot\nabla)\varphi]] \, \mathbb{Z}^{i,k}_{st}.
\end{gather*}

\begin{Remark}
As shown in \cite[Section 4.1.2]{Hofmanova_etal_2019}, the pressure term $\pi$ can be uniquely determined by the velocity $u$. Hence,  we only concentrate on the construction of the random dynamical system  associated  to \eqref{NSDiffFormSystem_u}.
\end{Remark}
To define the considered notion of a weak solution which is suitable for semiflow selection, let us first define the set of admissible  initial data
\begin{equation*}
\bD := \left\{ [x, e] \in \bH^0 \times \bR_+ : \frac{1}{2} | x |_0^2 \leq e \right\}.
\end{equation*}
Note that $\bD$ is a closed convex subset of $\bH^0 \times \bR_+$.  Recall that $\T = [0,\infty)$.
\begin{Definition}[Weak solution]\label{def-weaksolution}
	Given $[u_0, E_0] \in \bD $ and a geometric $\alpha$-H\"older rough path 
	\begin{equation}\label{p-var-rp}
	\bZ=(Z, \mathbb{Z}) \in C^{\alpha}_{2,\loc} (\T;\bR^{K}) \times C^{2\alpha}_{2,\loc} (\T; \bR^{K\times K}), \textrm{ for some } \alpha \in \left(\frac{1}{3}, \frac{1}{2} \right], 
	\end{equation} 
	we say that a pair $[u,E]$  is  a weak solution of \eqref{NSDiffFormSystem_u} if 
	\begin{enumerate}
		\item $u  : \T \rightarrow \bH^0 $ is a  weakly continuous function  and $u\in L^2_{\loc}(\T;\bH^1 ) \cap L^{\infty}_{\loc}(\T;\bH^0)$;
		\item $E: \T \to \bR_+$ satisfies $E(t) = \frac{1}{2} |u_t|_0^2$ a.e. $t \in \T$;
		\item $E(t)$ is a non-increasing function of $t$. In the variational form we write this as $E(0-) = E(0)$ and 
		\begin{align}\label{ineq-energy}
			\left[ E \psi \right]_{t=\tau_1-}^{t=\tau_2+}  - \int_{\tau_1}^{\tau_2} E \partial_t \psi \, dt  + \int_{\tau_1}^{\tau_2}  \psi \int_{\bT^3} |\nabla u_t|^2  \, dx \, dt \leq 0, 
		\end{align}
		for every $0 \leq \tau_1 \leq \tau_2 $ and $\psi \in C_c^1(\T)$ with $\psi \geq 0$. 
		\item 
		the remainder $u^{P,\natural} : \Delta_{\T} \rightarrow \bH^{-3}$ 	which is defined, for all $ \phi\in \bH^{3}$,  and 
		$(s,t)\in \Delta_{\T}$ by 
		\begin{align}
		u_{st}^{P,\natural}(\phi)& :=   \delta u_{st} (\phi ) +  \int_s^t \left[ (\nabla u_r, \nabla   \phi) + B_P(u_r)(\phi) \right]\,dr   -  u_s([A_{st}^{P,1,*} +A_{st}^{P,2,*}]\phi) , \label{SystemSolutionU}
		\end{align}
		satisfy
		\begin{equation} \label{SystemSolutionRemainder}
		u^{P,\natural} \in C^{\frac{p}{3}-\textnormal{var}}_{2, \varpi,L,\loc}(\T; \bH^{-3}) ,
		\end{equation}
		for some control $\varpi$ and $L> 0$. 
	\end{enumerate}
\end{Definition}
The next result gives existence of a  weak solution to \eqref{NSDiffFormSystem_u} for any initial data and a rough transport perturbation. Even though the energy inequality \eqref{ineq-energy} was not included in the corresponding definition of weak solution in \cite{Hofmanova_etal_2019}, it can be verified that it is satisfied  by  the solutions constructed in \cite[Theorem 2.13]{Hofmanova_etal_2019}. The necessary ideas are also discussed in the proof of stability in Theorem~\ref{thm-SeqStabilitywrtZandD} below.

\begin{Theorem}\cite[Theorem 2.13]{Hofmanova_etal_2019}\label{thm_existence}
	For a given initial data $[u_0,E_0] \in \bD$, a geometric $\alpha$-H\"older rough path  $\bZ$, for some $\alpha \in \left(\frac{1}{3}, \frac{1}{2} \right]$, there exists a  weak solution to \eqref{NSDiffFormSystem_u}, in the sense of Definition~\ref{def-weaksolution}, which satisfies,
	\begin{equation}\label{est-energy}
	 \frac{1}{2}|u_t|^2_0 + \int_0^t  |\nabla u_r|_0^2\,dr \leq \frac{1}{2} |u_0|_{0}^2 \leq E_0, \qquad \forall t \in \T. 
	\end{equation}
\end{Theorem}
\begin{Remark}
	Given $[u_0, E_0] \in \bD $ and a geometric $\alpha$-H\"older rough path $\bZ$, if $[u,E]$  is a weak solution of \eqref{NSDiffFormSystem_u}, then we can always consider that $\frac{1}{2} |u_t|_0^2 \le E(t-)$ for all $t \in \T$. Moreover, we will  write $\{[u,E](t); t \in \T \}$ and $\{[u_t,E(t-)]; t \in \T \}$ instead $[u,E]$ if we want to give information about the time scale. 
\end{Remark}

\section{Semiflow Selection}\label{sec:semiflow}

Throughout this section, we continue our pathwise analysis and let
\begin{equation*}
{\bf Z}  = (Z, \mathbb{Z}) \in \mathcal{C}^{\alpha}_{g,\loc}(\T;\bR^K)\,, \quad \alpha \in  \left( \frac{1}{3},\frac{1}{2}\right]\,,
\end{equation*}
be a continuous geometric $\alpha$-H\"older rough path. Randomness is only going to reappear in Section~\ref{sec:rds} below.
We consider the following separable metric space as the trajectory space 
\begin{equation}
\bX := C_{\loc}(\T; \bH^{-1}) \times L_{\loc}^1(\T; \bR). 
\end{equation}
For  data $[x, e] \in \bD$, we introduce the solution set
\begin{equation*} 
\begin{split}
\clU &[s,x, e, \bZ] : =
\left\{ [u, E] \in \bX \ \Big| \begin{array}{c}
[u, E] \ \mbox{is a weak solution to \eqref{NSDiffFormSystem_u} perturbed by $\bZ$} \\ \mbox{with initial time $s$ and initial data } [x, e]\end{array} \right\}.
\end{split}
\end{equation*}
If $s=0$ then to shorten the notation we write $\clU [x, e, \bZ]$ instead of $\clU[0,x, e, \bZ]$.

In order to fulfill the criterion of maximal energy dissipation, and following \cite{Breit_etal_2020}, for a fixed initial data and a rough path, we focus on a subclass of weak solutions consisting of the ones which minimize the total energy. To define this subclass we introduce a partial relation $\prec$ as follows:
if $[u^i, E^i]$, $i=1,2$, are two weak solutions to \eqref{NSDiffFormSystem_u} perturbed by the same rough path $\bZ$ and starting from the same initial data $[u_0, E_0]$, we write $[u^1, E^1] \prec [u^2, E^2]$ iff  
\begin{equation*} 
\ E^{1}(t\pm) \leq E^{2}(t\pm)
\ \mbox{for every}\ t \in \T \setminus \{0\}.
\end{equation*}

\begin{Definition}[Admissible weak solution] \label{defn-admissible}	
	We say that a weak solution $[u, E]$ to \eqref{NSDiffFormSystem_u} perturbed by  $\bZ$ starting from the initial data $[u_0, E_0]$ is \emph{admissible} if
	it is minimal with respect to the relation $\prec$. Specifically, if
	\[
	[\tilde{u}, \tilde{E}] \prec [u, E],
	\]
	where $[\tilde{u}, \tilde{E}]$ is another weak solution to \eqref{NSDiffFormSystem_u} driven by  path $\bZ$ and starting from $[u_0, E_0]$, then
	\[
	E = \tilde E  \ \mbox{on}\ \T.
	\]
\end{Definition}

We can now define a semiflow selection to \eqref{NSDiffFormSystem_u}.
\begin{Definition}[Semiflow selection] \label{defn-semiflowSelection}
	 A \emph{semiflow selection} in the class of weak solutions for the problem \eqref{NSDiffFormSystem_u} is a Borel measurable mapping
	\[
	U: \bD \times \mathcal{C}^{\alpha}_{g,\loc}(\T;\bR^K) \to \bX,\quad
	U\left\{ u_0, E_0 ,\bZ\right\} \in \clU [u_0, E_0,\bZ] \ \mbox{for any}\
	[u_0, E_0,\bZ] \in \bD \times \mathcal{C}^{\alpha}_{g,\loc}(\T;\bR^K) ,
	\]
	which enjoys the following \textbf{semigroup property}:
	\[
	U \left\{ u_0, E_0,\bZ \right\} (t_1 + t_2) =
	U \left\{ U\{u_{0},E_{0},\bZ\}(t_{1}) ,\tilde{\bZ}_{t_1} \right\}(t_2),
	\]
	for any $[u_0, E_0] \in \bD$ and any $t_1, t_2\in \T $, where  $\tilde{\bZ}_{t_1}(\cdot) := \bZ(t_1+\cdot)=(Z,\mathbb{Z})_{t_{1}+\cdot,t_{1}+\cdot}$.	
\end{Definition}

Observe that $\tilde{\bZ}_{t_1}(\cdot)$ defined above is again a rough path, in particular,  Chen's relation holds true.

\subsection{Sequential stability}\label{sec:SeqStability}

In this subsection we address the issue of sequential stability which will allow us to show the compactness of the set $\clU[{u_0,E_0,\bZ}]$ as well as the required measurability of the semiflow selection. It is also essential for proving the measurability of the random dynamical system constructed in Section~\ref{sec:rds}.

Given $T>0$, we let $\Delta_T := \Delta_{[0,T]}$ and $\Delta^{(2)}_T=\Delta^{(2)}_{[0,T]}$.

\begin{Theorem} \label{thm-SeqStabilitywrtZandD}
	Let $\{ \bZ^N = (Z^N, \Z^N) \}_{N \in \bN}$ be a sequence of geometric $\alpha$-H\"older rough paths such that $\bZ^N$ converges to some  $\alpha$-H\"older rough path $\bZ = (Z,\Z)$ in the product topology on $C_{2,\rm loc}^{\alpha}(\T; \bR^K) \times C_{2,\rm loc}^{2\alpha}(\T; \bR^{K\times K})$. Suppose that  $\{ [u_{0}^{N}, E_{0}^{N}] \}_{N \in \bN} \subset \bD$  is a sequence  of initial data and that there exists a positive real number $\clE$, such that 
	\begin{equation}\label{rp-bounds}
		E_{0}^{N} \leq \clE ,
	\end{equation}
	for every $N \in \bN.$
	 Let $ [u^{N}, E^{N}]   \in  \clU [u_{0}^{N}, E_{0}^{N}, \bZ^N ]$, $N\in\bN$, be a family of associated weak solutions.  Then 
	\begin{enumerate}
		\item there exist $[u_0,E_0] \in \bD$ and a subsequence, indexed again by $N$, such that
		\begin{equation} \label{ES1}
		u_{0}^{N} \to u_0 \ \mbox{weakly in}\ \bH^0, \qquad E_{0}^{N} \to E_0. 
		\end{equation}
		\item for the subsequence of solutions $\{ [u^{N}, E^{N}] \}_{N \in \bN} $, corresponding to the data $\{ [u_{0}^{N}, E_{0}^{N}] , \bZ^N \}_{N \in \bN} $ from part (a),  there exists a weak solution $[u,E] $ such that the following hold
		\[
		\begin{split}
		u^{N} &\to u \ \mbox{in}\ C_{\rm{loc}}(\T; \bH^{-1}),\\
		E^{N}(t) &\to E(t) \ \mbox{for any } t \in  \T \ \mbox{and in}\ L_{\rm{loc}}^1(\T; \bR).
		\end{split}
		\]
	\end{enumerate}
\end{Theorem}
\begin{proof}[\textbf{Proof of Theorem \ref{thm-SeqStabilitywrtZandD}}]
	First observe that the convergences in \eqref{ES1} follow immediately from the fact that $E_{0}^{N}$, in particular $|u_0^N|_0^2$, are uniformly bounded in $N$ by $\clE$. 
	 Let us fix an arbitrary time $T >0$. Notice that to prove the Theorem \ref{thm-SeqStabilitywrtZandD}, it is sufficient to prove all the required results on $[0,T]$. 

	Observe that due to convergence $\bZ^N \to \bZ$ in the mentioned product topology, for every $\varepsilon >0$, there exists an $N_0 := N_0(\varepsilon) \in \bN$ such that  \begin{align*}
	|||\bZ^N - \bZ|||_{\alpha,[0,T]} < \varepsilon , \quad \textrm{ for all } \quad N  \ge N_0. 
	\end{align*}
	Consequently, the reverse triangle inequality yields,  for all $N \ge N_0$, 
	\begin{align*}
	|||\bZ^N |||_{\alpha,[0,T]} < \varepsilon + |||\bZ |||_\alpha. 
	\end{align*}
	Since the above holds for every $\varepsilon >0$, we fix $\varepsilon =1$ and get
	\begin{align*}
	& |||\bZ^N |||_{\alpha,[0,T]}   \le \max\{ |||\bZ |||_{\alpha,[0,T]} +1, |||\bZ^1 |||_{\alpha,[0,T]}, \cdots, |||\bZ^{N_0} |||_{\alpha,[0,T]} \} =: R. 
	\end{align*}
	Let us set 
	\begin{equation}\label{ZcontrolDefn}
		\omega_Z(s,t) := (t-s)R^{1/\alpha}, \qquad (s,t) \in \Delta_T. 
	\end{equation}
	Then, it is easy to show that $\omega_Z$ is a control and we have
	\begin{align}\label{eqn-rp_bound}
	& |Z_{st}^N| \le (\omega_Z(s,t))^{\alpha}, \qquad |\Z_{st}^N|  \le (\omega_Z(s,t))^{2\alpha}, \quad \forall (s,t) \in \Delta_T.  
	\end{align}	

	\noindent To move further, let us define 
	\begin{align}\label{UnboundedOperSeq}
		& A^{N,1}_{st} \phi :=  P  \left[(\sigma_k \cdot \nabla) \phi \right]  Z_{st}^{N,k},  \nonumber\\
		& A^{N,2}_{st} \phi  :=  P  \left[ (\sigma_k \cdot  \nabla) P [ (\sigma_j \cdot  \nabla)  \phi ] \right] \mathbb{Z}_{st}^{N, j,k}.
	\end{align} 
	Next we claim that, for $\beta \in [0,2]$,  
	\begin{align}
		& |A^{N,1}_{st}  |_{\mathcal{L}(\bH^{\beta+1}, \bH^\beta)} \le M  (\omega_Z(s,t))^\alpha,  \quad \textrm{ for } \beta \in [0,2], \label{A1Nbound}\\
		& |A^{N,2}_{st}  |_{\mathcal{L}(\bH^{\beta+2}, \bH^\beta)}  \le M (\omega_Z(s,t))^{2\alpha},  \quad \textrm{ for } \beta \in [0,1], \label{A2Nbound}
	\end{align}
	where $M$ was introduced in Section~\ref{sec:HelmholtzLeray_projection}.
We will only prove \eqref{A1Nbound}, since the proof of \eqref{A2Nbound} is similar. Observe that, for $\beta \in [0,2]$, estimates \eqref{eqn-rp_bound} and  \eqref{A12bound-1} give 
	\begin{align*}
	|A^{N,1}_{st}  |_{\mathcal{L}(\bH^{\beta+1}, \bH^\beta)} &  \le |P \clA^1|_{\clL(\bH^{\beta+1},\bH^\beta)}   |Z_{st}^{N}|    \le  M   (\omega_Z(s,t))^\alpha. 
	\end{align*}
	Similarly we can get the  inequality in \eqref{A2Nbound}.  
	Hence, by Definition \ref{def:urd}, the $\{ (A^{N,1}_{st} , A^{N,2}_{st} ) \}_N$ is a family of unbounded rough drivers, with $$\omega_{A^N}(s,t) := M^{1/\alpha}  \omega_Z(s,t),$$ on the scale $(\bH^\beta)_{\beta \in \bR_+}$ which is uniformly bounded in $N$.

	Now, we  prove the convergences in part (b). The case when the sequence of rough paths does not depend on $N$ is proven \cite[Theorem~ 4.1]{Hofmanova_etal_2019}.  We still include the  whole idea here with more details for the completion.  
	
	First observe that, without loss of generality, we may assume that the same control $\varpi$ and constant $L>0$ works for each element of sequence $\{u^{N}\}_{N\geq 1}$ in the Definition \ref{def-weaksolution}.
	Since, for every $N \in \bN$,  $\frac{1}{2}|u_0^N|_0^2 \le \clE $ and corresponding $u^N$ satisfies the energy inequality \eqref{est-energy}, we get that the sequence $\{u^{N}\}_{N \geq 1}$ is uniformly bounded in $L^2_T\bH^1\cap L^{\infty}_T\bH^0$, an application of Banach-Alaoglu theorem yields a subsequence, which we will index again as $\{u^{N}\}_{N\geq 1}$, that converges   weakly in $L^2_T\bH^1$ and weak-* in $L^{\infty}_T\bH^0$.

	To obtain a further subsequence that converges strongly in $L_T^{2}\bH^0 \cap C_T\bH^{-1}$, thanks to Lemma~\ref{CompactnessLemma}, it is sufficient to show that there exist controls $\omega$ and $\bar{\omega}$ and $\bar{L},\kappa>0$,  independent of $N$, such that  $|\delta u_{st}^{N}|_{-1}\le \omega(s,t)$ for all $(s,t)\in \Delta_T$ with $\bar{\omega}(s,t)\le \bar{L}$.

	Let $\phi \in \bH^1$. Decomposition of  $\delta u^N_{st}$ into a smooth and non-smooth part using $J^{\eta}$ (defined in Section~\ref{sec:SmoothingOper}) for some $\eta\in (0,1]$, yields 
	\begin{equation}\label{SSwrtZandD-1}
	|\delta u^N_{st}(\phi)| \leq |\delta u^N_{st}( J^{\eta} \phi)| + |\delta u^N_{st}( (I - J^{\eta}) \phi)| .
	\end{equation}
	By applying \eqref{smoothingOperator} and \eqref{est-energy} we estimate the second term in above as
	\begin{equation}\label{SSwrtZandD-2}
	|\delta u^N_{st}( (I - J^{\eta}) \phi)| \les  |u^N|_{L_T^\infty\bH^0} |(I - J^{\eta}) \phi|_0 \les \eta |u_0^{N}|_0 |\phi|_1 \leq  \eta \sqrt{\clE }|\phi|_1. 
	\end{equation}
	For the first term on the  right hand side of \eqref{SSwrtZandD-1}, by letting
	$$
	\mu_{t}^N(\phi) :=  - \int_0^t \left[ (\nabla u_r^N,\nabla \phi) +B_P(u_r^N)(\phi) \right]\,dr, \quad \phi \in \bH^1,
	$$ 
	\eqref{SystemSolutionU} gives that for all $(s,t)\in \Delta_T$,  
	\begin{equation}\label{RNS2}
	\delta u_{st}^N =  \delta \mu_{st}^N + A_{st}^{N,1} u_s+ A_{st}^{N,2}u_s + u^{P,\natural,N}_{st},
	\end{equation}
	where the equality holds in $\bH^{-3}$. 
	Consequently, we get
	\begin{equation}\label{SSwrtZandD-3}
	|\delta u^N_{st}( J^{\eta} \phi)| \leq  |u^{P,\natural,N}_{st} ( J^{\eta} \phi)|  + |\delta \mu_{st}^{N} ( J^{\eta} \phi)| + | u_s^{N} (A_{st}^{N,1,\ast}   J^{\eta} \phi) | + |u_s^{N} (A_{st}^{N,2,\ast} J^{\eta} \phi)|. 
	\end{equation}
	We estimate each term of \eqref{SSwrtZandD-3} separately as follows: 
	
	1) By Lemma \ref{Thm2.5}  we infer that there is a positive constant $\tilde{L}$, depending only on $p$ (i.e., independent of $N$),  such that  for all $(s,t)\in \Delta_T$ with  $\varpi(s,t)\le L$ and $ M^{1/\alpha}\omega_Z(s,t) = \omega_{A^N}(s,t)\leq \tilde{L}$, the following inequality is true 
	\begin{align} \label{SSwrtZandD-7}
	\omega_{P, \natural,N}(s,t) &  \les_{p} |u^N|^{\frac{p}{3}}_{L^{\infty}_T\bH^0}  \omega_{A^N}(s,t) + ( 1 + |u^N|_{L^{\infty}_T\bH^0} )^{\frac{2p}{3}}(t-s)^{\frac{p}{3}} \omega_{A^N}(s,t)^{\frac{1}{12}} \nonumber\\
	&  \les_{p} M^{1/\alpha} |u^N|^{\frac{p}{3}}_{L^{\infty}_T\bH^0}  \omega_{Z}(s,t) + M^{1/\alpha} ( 1 + |u^N|_{L^{\infty}_T\bH^0} )^{\frac{2p}{3}}(t-s)^{\frac{p}{3}} \omega_{Z}(s,t)^{\frac{1}{12}},
	\end{align}
	where $\omega_{P, \natural,N}(s,t) := |u^{P, \natural,N}|^{\frac{p}{3}}_{\frac{p}{3} -var; [s,t];\bH^{-3}}.$
	Since $u^{P,\natural,N}$ is the remainder, \eqref{SystemSolutionRemainder}  followed by \eqref{smoothingOperator} and \eqref{SSwrtZandD-7} yield
	\begin{align}\label{SSwrtZandD-8}
	|u^{P,\natural,\varepsilon}_{st} ( J^{\eta} \phi)| & \leq \omega_{P,\natural,N}(s,t)^{\frac{3}{p}} |J^{\eta} \phi |_3 \nonumber\\
	&  \les_p \eta^{-2} M^{\frac{3}{p \alpha}} \left[ |u^N|^{\frac{p}{3}}_{L^{\infty}_T\bH^0}  \omega_{Z}(s,t) + ( 1 + |u^N|_{L^{\infty}_T\bH^0} )^{\frac{2p}{3}}(t-s)^{\frac{p}{3}} \omega_{Z}(s,t)^{\frac{1}{12}} \right]^{\frac{3}{p}} |\phi |_1 \nonumber \\
	& \les \eta^{-2} \left[ \sqrt{\clE}  (\omega_{Z}(s,t))^{\frac{3}{p}}  + ( 1 + \clE ) (t-s) (\omega_{Z}(s,t))^{\frac{1}{4p}} \right] |\phi |_1. 
	\end{align}

	2) Let us set $\omega_{\mu^N}(s,t) :=  \int_s^t  (1 + |u_r^N|_1)^2 \, dr$. 	Since $u^N \in L_T^2 \bH^1$, we infer that $\omega_{\mu}$ is a control. Using  \eqref{trilinear form estimate} and the the Cauchy-Schwartz inequality we deduce that 
	\begin{align*}
	|\delta \mu_{st} |_{-3} & \le \sup_{|\phi|_3 \le 1} \int_s^t \left[ | (\nabla u_r,\nabla \phi)| +|B_P(u_r)(\phi)| \right] \,dr  \\
	& \le \sup_{|\phi|_3 \le 1}   | \phi|_3 \int_s^t |\nabla u_r|_0  \,dr +  \int_s^t  |u_r|_1^2 \,dr  \les \int_s^t \left( 1 + 2| u_r|_1 + | u_r|_1^2  \right)  \,dr   = \omega_{\mu}(s,t).
	\end{align*}
	Then, by using  \eqref{smoothingOperator} and \eqref{est-energy} we obtain
	\begin{align}\label{SSwrtZandD-9}
	|\delta \mu_{st}^{N} ( J^{\eta} \phi)|& \les  \eta^{-2} \omega_{\mu^N}(s,t) |\phi |_1 = \eta^{-2} \left[  \int_s^t  (1 + |u_r^N|_1)^2 \, dr \right] |\phi |_1 \les \eta^{-2} (t-s) (1+\clE)  |\phi |_1. 
	\end{align}

	3) By applying \eqref{A1Nbound}  followed by  \eqref{smoothingOperator} and \eqref{est-energy} we yield, with $\alpha = \frac{1}{p}$, 
	\begin{align}\label{SSwrtZandD-10}
	| u_s^{N} (A_{st}^{N,1,\ast}   J^{\eta} \phi) |  & \les M  |u^N|_{L_T^\infty\bH^0}  (\omega_{Z}(s,t))^{\frac{1}{p}}   | J^{\eta} \phi|_1 \les   |u_0^N|_0  (\omega_{Z}(s,t))^{\frac{1}{p}}  | \phi |_1 \nonumber \\
	& \leq  \sqrt{\clE}  (\omega_{Z}(s,t))^{\frac{1}{p}}  | \phi |_1.
	\end{align}

	4)  Again by  using  \eqref{A2Nbound} with \eqref{smoothingOperator}  and \eqref{est-energy} we get
	\begin{align}\label{SSwrtZandD-11}
	| u_s^{\varepsilon} (A_{st}^{N,2,\ast}   J^{\eta} \phi) |  & \les M |u^N|_{L_T^\infty\bH^0}  (\omega_{Z}(s,t))^{\frac{2}{p}}   | J^{\eta} \phi|_2 \les \eta^{-1} |u_0^N|_0  (\omega_{Z}(s,t))^{\frac{2}{p}}   | \phi |_1\nonumber  \\
	& \leq  \eta^{-1} \sqrt{\clE } (\omega_{Z}(s,t))^{\frac{2}{p}}   | \phi |_1.  
	\end{align}
	So by substituting \eqref{SSwrtZandD-2}, \eqref{SSwrtZandD-8}-\eqref{SSwrtZandD-11} into \eqref{SSwrtZandD-3}, for each $\phi \in \bH^1$ we have
	\begin{align*}
	|\delta u^N_{st}(\phi)| & \les   \eta^{-2} \left[ \sqrt{\clE}  (\omega_{Z}(s,t))^{\frac{3}{p}}  + ( 1 + \clE ) (t-s) (\omega_{Z}(s,t))^{\frac{1}{4p}} \right] |\phi |_1 + \eta^{-1} (t-s) (1+\clE)  |\phi |_1 \\
	& \quad  +  \sqrt{\clE}  (\omega_{Z}(s,t))^{\frac{1}{p}}  | \phi |_1 + \eta^{-1} \sqrt{\clE } (\omega_{Z}(s,t))^{\frac{2}{p}}   | \phi |_1 + \eta \sqrt{\clE} |\phi|_1.
	\end{align*}
	Let us set $\eta := (\omega_{Z}(s,t))^{\frac{1}{p}} + (t-s)^{\frac{1}{p}} $. Observe that we can choose $M$ such that $\eta \in [0,1)$. Indeed, since $\tilde{L}$  is fixed and $\omega_Z(s,t) $ defines as in \eqref{ZcontrolDefn}, we can choose $M$ large enough such that the inequalities \eqref{A12bound-1}-\eqref{A12bound-2},  the relation  $\omega_Z(s,t) \le  \frac{\tilde{L}}{M^{1/\alpha}} < \frac{1}{2}$ and $(t-s)^{\frac{1}{p}} < \frac{1}{2}$ hold true for all $(s,t)\in \Delta_T$ with $\bar{\omega}(s,t)\le \bar{L}$.

	Consequently, we infer that
	\begin{align}
	|\delta u^N_{st}|_{-1} &  \les_{M, \clE}(1+ |u_0|_0)^2 ( \omega_{Z}(s,t)^{ \frac{1}{p} } + (t-s)^{1 - \frac{2}{p}} )  \label{uniformSolution}  .
	\end{align}
	Since for $p \geq 2$, and $\kappa >0$ 
	\begin{align*}
	\omega_{Z}(s,t)^{ \frac{1}{p} } + (t-s)^{1 - \frac{2}{p}}  \les_{p,\kappa}  \left( \omega_{Z}(s,t)^{ \frac{\kappa}{p} } + (t-s)^{ \kappa \left( 1 - \frac{2}{p}\right)}  \right)^{\frac{1}{\kappa}},
	\end{align*}
	by choosing $\kappa$ which satisfy
	$$ \kappa \geq p \textrm{ and } \kappa \geq \frac{p}{p-2}, $$
	we deduce  that $$ \tilde{\omega}(s,t) :=  \omega_{Z}(s,t)^{ \frac{\kappa}{p} } + (t-s)^{ \kappa \left( 1 - \frac{2}{p}\right)}, $$  is a control. 
	
	Hence, due to Compactness Lemma \ref{CompactnessLemma}, there is a subsequence of  $\{u^N\}_{N\in \bN}$, which we continue to denote by $\{u^N\}_{N \in \bN}$, converging strongly to an element $u $ in $C_T\bH^{-1}\cap L^2_T\bH^0$. 
	Now recall that $\delta u_{st}^{N}  := u_t^{N} -u_s^{N}$  and, from Definition \ref{def-weaksolution},
	\begin{equation} \label{NSRoughFormSystem}
	\begin{aligned}
	\delta u_{st}^{N} +\int_s^t  B_p(u_r^N)\,dr&=\int_s^t \Delta u_r^{N} dr+ [A_{st}^{N,1}+A_{st}^{N,2}]u_s^{N} +u_{st}^{P, \natural,N} ,
	\end{aligned}
	\end{equation}
	where $A^{N,1}_{st}$ and $A^{N,2}_{st}$ are defined as in \eqref{UnboundedOperSeq}. 

	Our goal now is to prove that $u$ is a weak solution to \eqref{NSDiffFormSystem_u}. The idea is to pass the limit in \eqref{NSRoughFormSystem} tested against some $\phi \in \bH^3$ as $N$ tends to $\infty$. 
	
	For the terms with operators $A^{N,i}, i=1,2$,  observe that
	\begin{align}\label{SSwrtZandD-15}
	& |(u^N_s,A_{st}^{N,i,*}\phi) - (u_s,A_{st}^{i,*}\phi)|  \leq | u^N_s-u_s |_{-1} |A_{st}^{N,i,*}\phi |_{1} + | u^N_s-u_s|_0 |(A_{st}^{N,i,*} - A_{st}^{i,*}) \phi |_0. 
	\end{align}
	The first term in the r.h.s of \eqref{SSwrtZandD-15} goes to $0$ as $N \to \infty$ because $ u^N \to u$ in $C_T\bH^{-1}$.  To estimate the second term in \eqref{SSwrtZandD-15} we proceed as follows: bound \eqref{A12bound-1} yield
	\begin{align}\label{SSwrtZandD-16}
	& |(A_{st}^{N,1,*} - A_{st}^{1,*}) \phi |_0   \leq  |P[\sigma_k \cdot \nabla]|_{\mathcal{L}(\bH^1,\bH^0)} |\phi|_1  |Z_{st}^{N,k} - Z_{st}^{k} |  \les_M  |\phi|_1  |Z_{st}^{N} - Z_{st} |.
	\end{align}
	Similarly, the estimate \eqref{A12bound-2}  gives
	\begin{align}\label{SSwrtZandD-17}
		& |(A_{st}^{N,2,*} - A_{st}^{2,*}) \phi |_0   \leq   |P[ (\sigma_k \cdot \nabla) P[\sigma_j \cdot \nabla]]|_{\mathcal{L}(\bH^2,\bH^0)} |\phi|_2  | \Z_{st}^{N,i,k} - \Z_{st}^{i,k} |  \les_{M}  |\phi|_2  |\Z_{st}^{N} - \Z_{st} |.
	\end{align}
	So, since $Z^N \to Z$ in $C^{\alpha}_{2,\loc}(\T; \bR^K) $ and $\Z^N \to \Z$ in $ C^{2\alpha}_{2,\loc}(\T;\bR^{K \times K})$, from 	\eqref{SSwrtZandD-16}-\eqref{SSwrtZandD-17} we infer that, for $i=1,2$, 
	\begin{equation}\label{SSwrtZandD-18}
		|(u^N_s,A_{st}^{N,i,*}\phi) - (u_s,A_{st}^{i,*}\phi)|  \to 0 \textrm{ as } N \to \infty. 
	\end{equation}
	Further, using the  H\"older inequality, the strong convergence in $L_T^2\bH^0$ of $\{u^N\}_{\varepsilon >0}$ and trilinear estimate \eqref{trilinear form estimate},  we find
	\begin{align*}
	& \left|\int_s^t \left[B_P(u_r)(\phi) - B_P(u_r^N) (\phi) \right]\,dr \right| \\
	&  \leq \left| \int_s^t B_P(u_r - u_r^N,u_r)(\phi)\,dr \right| 	+ \left|\int_s^t  B_P(u_r^N,u_r- u_r^N) (\phi) \, dr \right| \\
	&  \les  \int_s^t |u_r - u_r^N|_0 |u_r|_0 \, dr |\phi|_3    +  \int_s^t |u_r - u_r^N|_0 |u^N_r|_0 \, dr |\phi|_3 \\
	&  \leq  |\phi|_3 \left( \int_s^t |u_r - u_r^N|_0^2 \, dr  \right)^{1/2} \left[ \left( \int_s^t |u_r|_0^2 \, dr  \right)^{1/2}     + \left(   \int_s^t |u_r^N|_0^2 \, dr  \right)^{1/2} \right]   \rightarrow 0,
	\end{align*}
	as $N \to \infty$.
	Finally, using the  H\"older inequality, the strong convergence in $L_T^2\bH^0$ of $\{u^N\}_{\varepsilon >0}$ we have 
	\begin{align*}
	& \left|\int_s^t [\Delta u_r^N - \Delta u_r ] (\phi) \,dr \right|  \leq   |\Delta \phi|_0 \int_s^t |u_r^N - u_r|_0  \, dr \le  |\Delta \phi|_0  \left( \int_s^t |u_r - u_r^N|_0^2 \, dr  \right)^{1/2} |t-s|^{1/2}	 \rightarrow 0,
	\end{align*}
	as $N \to \infty$.

	Hence, since we have shown that all of the terms in equation \eqref{NSRoughFormSystem} converge when applied to $\phi$,  the remainder $u^{P, \natural,N}_{st}(\phi)$ converges to some limit $u^{P, \natural}_{st}(\phi)$. Since $|u^{P, \natural}|^{\frac{p}{3}}_{\frac{p}{3} -var; [s,t];\bH^{-3}}$ is equal to the infimum over all controls satisfying $|u_{st}^{P, \natural }|_{-3} \leq \omega_{P, \natural }(s,t)^{\frac{3}{p}}$, by above convergence results and \eqref{SSwrtZandD-7} we obtain
	\begin{align*}
	|u_{st}^{P, \natural }|_{-3} &  \leq \sup_{|\phi|_{3} \leq 1}  |(u^{P,\natural}_{st}-u^{P,\natural,N}_{st})(\phi)|  + \sup_{|\phi|_{3} \leq 1}  |u^{P,\natural,N}_{st}(\phi)|,
	\end{align*}
	where, as  in \eqref{SSwrtZandD-8},
	\begin{align*}
	|u^{P,\natural,N}_{st}(\phi)| &   \les \left[ \clE  (\omega_{Z}(s,t))^{\frac{3}{p}}  + ( 1 + \clE ) (t-s) (\omega_{Z}(s,t))^{\frac{1}{4p}} \right]   | \phi |_3.
	\end{align*}
	So, by taking the limit $N \to \infty$ we get that
	\begin{align*}
	|u_{st}^{P, \natural }|_{-3} &   \les \left[ \clE  (\omega_{Z}(s,t))^{\frac{3}{p}}  + ( 1 + \clE ) (t-s) (\omega_{Z}(s,t))^{\frac{1}{4p}} \right]. 
	\end{align*}
	 Hence, $ u^{P,\natural} \in C^{\frac{p}{3}-\textnormal{var}}_{2, \varpi,L}([0,T]; \bH^{-3})$ for some control $\varpi$ depending only on $\omega_Z$ and $L> 0$ depending only on $p$.

	Next, we prove that $u \in C_T\bH_w^0$. Recall that $u \in L_T^\infty\bH^0 \cap C_T\bH^{-1}$. Let $\phi \in \bH^0$. Since $\bH^1$ is dense in $\bH^0$, there exists a sequence $\{ \phi_n \}_{n \in \N} \subset \bH^1$ such that $|\phi_n - \phi|_0 \to 0$ as $n  \to \infty$. Then 
	\begin{align*}
	|\langle u_t-u_s, \phi \rangle |_0  \leq  2 |u|_{L_T^\infty\bH^0} | \phi - \phi_n|_0 + |\langle u_t-u_s, \phi_n \rangle |_0 .
	\end{align*}
	Since $u \in C_T\bH^{-1}$ and $\phi_n \in \bH^1$, $|\langle u_t-u_s, \phi_n \rangle |_0  \to 0$ as $s \to t$. Consequently,\begin{equation*}
	\lim\limits_{s \to t}	|\langle u_t-u_s, \phi \rangle |_0  \leq 2 |u|_{L_T^\infty\bH^0} \delta
	\end{equation*}
	for any $\delta >0$. So $\lim\limits_{s \to t} u_s(\phi) = u_t(\phi), \forall \phi \in \bH^0$.

	Now note that, since $E^N$ is non-increasing, for every $T >0$ and any partition  $0 = t_0 < t_1 <\cdots < t_n =T$ we have
	\begin{align*}
	\sum_{i=0}^{n-1} |E^N(t_{i+1}) - E^N(t_{i})|  \leq \clE. 
	\end{align*}
	Consequently, we get that its total variation is uniformly bounded. Hence by the Helly selection theorem there exists a subsequence of $\{ E^N \}_{N \in \bN}$, which we index again by $N$, and a function $E: \T \to \bR_+$ locally of bounded variation such that
	\[
	E^N(t) \to E(t) \ \mbox{for any}\ t \in \T \ \mbox{and in}\ L^{1}_{\rm loc}(0,\infty). 
	\]
	But since $\{u^N \}_{N \in \bN}$ converges strongly to $u$ in $L_T^{2}\bH^0$, we infer that $E(t) = \frac{1}{2} | u_{t} |_0^2$ for a.e. $t \in \T$. 
	
	Hence, $[u,E]$ is a solution of \eqref{NSDiffFormSystem_u} in the sense of Definition \ref{def-weaksolution} and the proof of Theorem \ref{thm-SeqStabilitywrtZandD} completes. 
\end{proof}

\subsection{Shift invariance and continuation  property}\label{sec:ShiftInvariance_Continuation}
Here we prove the remaining two main ingredients for the construction of a semiflow, the shift invariance property and the continuation property of the set of solutions.

For $w \in \bX $, we define the positive shift operator $S_T \circ w$ as 
\[
S_T \circ w(t) := w(T + t),\ t \geq 0.
\]

\begin{Lemma}[Shift invariance property] \label{AL2}
	Let $[u_0, E_0] \in \bD$, $\bZ$ be a geometric $\alpha$-H\"older rough path defined on $\T$, and $[u, E] \in \clU[u_0,  E_0,\bZ]$.
	Then we have
	\[
	S_T \circ [u, E] \in \clU[u (T),  \mathcal{E}, \tilde{\bZ}_T]
	\]
	for any $T > 0$, and any $\mathcal{E} \geq E(T+)$. Here we recall the notation $\tilde{\bZ}_T(t) := \bZ(t+T)$ for $t \geq 0$. 
	
\end{Lemma}
\begin{proof}[\textbf{Proof of Lemma \ref{AL2}}]
	Let us fix $T >0$. Based on the definition of $S_T$, we need to show that
	\begin{equation*}
	(S_T \circ [u,E])(t) = \{ [u_{t+T}, E(t+T)]; t \geq 0 \} =: \{[\tilde{u}_t, \tilde{E}(t)]; t \geq 0 \} \in \clU[u(T), \mathcal{E}, \tilde{\bZ}]. 
	\end{equation*}
	Firstly, we observe that since $u \in\clU[u_0,  E_0,\bZ]$ it holds  $\tilde{u} \in L^2_{\loc}(\T;\bH^1 ) \cap L^{\infty}_{\loc}(\T;\bH^0)$.  Next, since $E(t)$ is a non-increasing function of $t$ and satisfies \eqref{ineq-energy}, we have
	\begin{align*}
	\left[ \tilde{E}(t) \psi(t) \right]_{t=\tau_1-}^{t=\tau_2+}  - \int_{\tau_1}^{\tau_2} \tilde{E}(t)\partial_t \psi(t) \, dt + \int_{\tau_1}^{\tau_2}  \psi \int_{\bT^3} |\nabla \tilde{u}_t|^2  \, dx \, dt \leq 0, \quad 0 \leq \tau_1 \leq \tau_2 ,  
	\end{align*}
	for every $\psi \in C_c^1(\T)$ with $\psi \geq 0$.

	Observe that, since  $\mathcal{E} \geq E(T+)$, $[u(T), \mathcal{E} ] \in \bD$ due to \eqref{est-energy}. Moreover,  \eqref{SystemSolutionU} gives	
	\begin{align}
	u_{(s+T)(t+T)}^{P,\natural}(\phi)& =   u_{t+T} (\phi ) - u_{s+T} (\phi )  +  \int_{s+T}^{t+T} \left[ (\nabla u_r, \nabla   \phi) + B_P(u_r)(\phi) \right]\,dr   \nonumber\\
	& \quad -  u_{s+T}([A_{(s+T)(t+T)}^{P,1,*} +A_{(s+T)(t+T)}^{P,2,*}]\phi) \nonumber\\
	& =  \tilde{u}_t (\phi ) - \tilde{u}_s (\phi )  +  \int_{s}^{t} \left[ (\nabla \tilde{u}_{\tilde{r}}, \nabla   \phi) + B_P(\tilde{u}_{\tilde{r}})(\phi) \right]\,d\tilde{r}  \nonumber\\
	& \quad -  \tilde{u}_s([A_{(s+T)(t+T)}^{P,1,*} +A_{(s+T)(t+T)}^{P,2,*}]\phi).
	\end{align}
	But, since under the notation $\tilde{Z}_{st} = Z_{(s+T)(t+T)}$ and  $\tilde{\Z}_{st} = \Z_{(s+T)(t+T)}$, we have
	\begin{align*}
	A_{(s+T)(t+T)}^{P,1}  = P[ (\sigma_k \cdot \nabla) \varphi] \tilde{Z}_{st}^k = \tilde{A}_{st}^{P,1},
	\end{align*}
	and
	\begin{align*}
	A_{(s+T)(t+T)}^{P,2}  & =  P[ (\sigma_k \cdot \nabla) P[ (\sigma_l \cdot \nabla) \varphi] ]\tilde{\Z}_{st}^{l,k} = \tilde{A}_{st}^{P,2}.
	\end{align*}
	Hence, for $(s,t)\in \Delta_T$,  \begin{align}
	u_{(s+T)(t+T)}^{P,\natural}(\phi)&   =  \tilde{u}_t (\phi ) - \tilde{u}_s (\phi )  +  \int_{s}^{t} \left[ (\nabla \tilde{u}_{\tilde{r}}, \nabla   \phi) + B_P(\tilde{u}_{\tilde{r}})(\phi) \right]\,d\tilde{r}  \nonumber\\
	& \quad -  \tilde{u}_s([\tilde{A}_{st}^{P,1,*} +\tilde{A}_{st}^{P,2,*}]\phi) =: \tilde{u}_{st}^{P,\natural}(\phi).
	\end{align}
	To finish the proof of Lemma \ref{AL2}, it remains to show that for every $\tau>0$, $\tilde{u}^{P,\natural} \in C^{\frac{p}{3}-\textnormal{var}}_{2, \varpi,L}([0,\tau]; \bH^{-3})$. But this we have since there exits a control $\tilde{w}_{\natural}$ such that  \begin{equation*}
	\Vert \tilde{u}_{st}^{P,\natural} \Vert_{-3} \leq c (\tilde{w}_{\natural}(s,t))^{\frac{3}{p}}, \quad \forall (s,t)\in \Delta_\tau. 
	\end{equation*}
	Indeed, since $u^{P,\natural} \in C^{\frac{p}{3}-\textnormal{var}}_{2, \varpi,L}([0,\tau]; \bH^{-3})$, there exits a control $w_{\natural}$ such that, for every $\phi \in \bH^{3}$, we have \begin{align*}
	\vert \tilde{u}_{st}^{P,\natural}(\phi) \vert =  \vert u_{(s+T)(t+T)}^{P,\natural}(\phi) \vert \leq c \Vert \phi\Vert_{\bH^{3}}  (w_{\natural}(s+T,t+T))^{\frac{3}{p}}. 
	\end{align*}
	Hence, by setting $\tilde{w}_{\natural}(s,t) := w_{\natural}(s+T,t+T)$, we get $\Vert \tilde{u}_{st}^{P,\natural} \Vert_{-3} \leq c (\tilde{w}_{\natural}(s,t))^{\frac{3}{p}}, \forall (s,t)\in \Delta_\tau $ and finishes the proof of Lemma \ref{AL2}. 
\end{proof}

For $w_1, w_2 \in \bX$ and $T >0$ we define the continuation operator $\omega_1 \cup_T \omega_2$ by
\[w_1 \cup_T w_2 (\tau) := \left\{
\begin{array}{l}
w_1(\tau) \ \mbox{for}\ 0 \leq \tau \leq T,\\ \\
w_2(\tau - T) \ \mbox{for}\ \tau > T. \end{array} \right.
\]

\begin{Lemma}[Continuation property] \label{AL3}
	
	Let $[u_0, E_0] \in \bD$, $\bZ$ be an $\alpha$-H\"older rough path, and
	\[
	[u,  E] \in \clU [u_0,  E_0,\bZ],
	\ [\tilde{u}, \tilde{E}] \in \clU [u(T),  \mathcal{E},\tilde{\bZ}] \ \quad \mbox{for some}\  \quad \mathcal{E} \leq E(T-).
	\]
	
	Then
	\[
	[u, E] \cup_T [\tilde{u}, \tilde{E}]  \in \clU [u_0,  E_0,\bZ].
	\]
	
\end{Lemma}

\begin{proof}[\textbf{Proof of Lemma \ref{AL3}}]
	Since the initial energy for $[\tilde{u}, \tilde{E}]$ is less or equal to $E(T-)$, we have that the energy of the solution $[u, E] \cup_T [\tilde{u}, \tilde{E}]$ indeed remains non-increasing
	on $\T$ and bounded by $E_0$ from above.
	
	Let us set $v := [u, E] \cup_T [\tilde{u}, \tilde{E}]$.  It remains to show that, for every $\tau >0$, 
	$$v^{P,\natural} \in C^{\frac{p}{3}-\textnormal{var}}_{2, \varpi,L,\loc}(\T; \bH^{-3}),$$ where 
	\begin{align*}
	v_{st}^{P,\natural}(\phi)& :=   v_{t} (\phi ) -  v_{s} (\phi )   +  \int_{s}^{t} \left[ (\nabla  v_{r}, \nabla   \phi) + B_P(v_r)(\phi) \right]\,dr   \nonumber\\
	& \quad -  v_{s}([A_{st}^{P,1,*} +A_{st}^{P,2,*}]\phi). 
	\end{align*}
	For this we will prove that there exists a control $\omega_{P,A,u,\natural} $ such that
	\begin{equation}
		\Vert v_{st}^{P,\natural} \Vert_{\bH^{-3}}  \les \left(  \omega_{P,A,u,\natural} (s,t) \right)^{\frac{3}{p}} , \quad \forall s <t. 
	\end{equation}
	Recall that, by definition of solution, for every $s <t$, 
	\begin{align*}
	u_{st}^{P,\natural}(\phi)& =   u_{t} (\phi ) -  u_{s} (\phi )   +  \int_{s}^{t} \left[ (\nabla  v_{r}, \nabla   \phi) + B_P(u_r)(\phi) \right]\,dr   \nonumber\\
	& \quad -  u_{s}([A_{st}^{P,1,*} +A_{st}^{P,2,*}]\phi), 
	\end{align*}
	and 
	\begin{align*}
	\tilde{u}_{st}^{P,\natural}(\phi)& =   \tilde{u}_{t} (\phi ) -  \tilde{u}_{s} (\phi )   +  \int_{s}^{t} \left[ (\nabla  \tilde{u}_{r}, \nabla   \phi) + B_P(\tilde{u}_r)(\phi) \right]\,dr   \nonumber\\
	& \quad -  \tilde{u}_{s}([\tilde{A}_{st}^{P,1,*} +\tilde{A}_{st}^{P,2,*}]\phi), 
	\end{align*}
	where 
	\begin{align*}
	& A_{st}^{P,1}  = P[ (\sigma_k \cdot \nabla) ] Z_{st}^k =: P_{1,k} Z_{st}^k, \quad A_{st}^{P,2}   = P[ (\sigma_k \cdot \nabla) P[ (\sigma_l \cdot \nabla) ] ]\Z_{st}^{l,k} =: P_{2,l,k} \Z_{st}^{l,k} \nonumber\\
	& \tilde{A}_{st}^{P,1}  = P_{1,k} \tilde{Z}_{st}^k, \quad \tilde{A}_{st}^{P,2}   = P_{2,l,k} \tilde{\Z}_{st}^{l,k}.   
	\end{align*}

	Note that the only interesting case is $s < T <t$ because if $s<t \leq T$ or $ s <t  \in (T,\infty)$ then only one out of $u$ or $\tilde{u}$ is active. Since $ u_{T} (\phi ) = \tilde{u}_{0} (\phi )$, for $s < T <t$ and $\phi \in \bH^3$, we  have
	
	\begin{align}
	& \delta v_{st}(\phi)   = \delta \tilde{u}_{0(t-T)}(\phi)  + \delta u_{sT}(\phi)  \nonumber\\
	&  =   \tilde{u}_{0}([\tilde{A}_{0(t-T)}^{P,1,*} +\tilde{A}_{0(t-T)}^{P,2,*}]\phi) - \int_{0}^{t-T} \left[ (\nabla \tilde{u}_{r}, \nabla   \phi) + B_P( \tilde{u}_{r})(\phi) \right]\,dr   + \tilde{u}_{0(t-T)}^{P,\natural}(\phi) \nonumber\\
	& \quad+  u_{s}([A_{sT}^{P,1,*} +A_{sT}^{P,2,*}]\phi) -  \int_{s}^{T} \left[ (\nabla  u_{r}, \nabla   \phi) + B_P(u_r)(\phi) \right]\,dr  + u_{sT}^{P,\natural}(\phi) \nonumber\\ 
	&  =   u_T([A_{Tt}^{P,1,*} +A_{Tt}^{P,2,*}]\phi) - \int_{s}^{t} \left[ (\nabla v_{r}, \nabla   \phi) + B_P(v_{r})(\phi) \right]\,dr   + \tilde{u}_{0(t-T)}^{P,\natural}(\phi) \nonumber\\
	& \quad+  u_{s}([A_{sT}^{P,1,*} +A_{sT}^{P,2,*}]\phi)  + u_{sT}^{P,\natural}(\phi).  \nonumber
	\end{align}
	Let us first observe that,  since $u_s = v_s$, due to Chen's relation $\Z_{st}^{l,k} - \Z_{sT}^{l,k} - \Z_{Tt}^{l,k} = Z_{sT}^l \otimes Z_{Tt}^k$,
	\begin{align*}
	& [A_{Tt}^{P,1} +A_{Tt}^{P,2}]u_T  +  [A_{sT}^{P,1} +A_{sT}^{P,2}]u_{s} \\
	& \quad = P_{1,k} u_T  Z_{Tt}^k +  P_{2,l,k} u_T  \Z_{Tt}^{l,k} + P_{1,k} u_s Z_{sT}^k +  P_{2,l,k} u_s  \Z_{sT}^{l,k} \\
	& \quad = P_{1,k} \delta u_{sT}  Z_{Tt}^k +  P_{2,l,k} \delta u_{sT}  \Z_{Tt}^{l,k} + P_{1,k} u_s  Z_{st}^k  +  P_{2,l,k} u_s  (\Z_{st}^{l,k} -  Z_{sT}^l \otimes Z_{Tt}^k) \\
	& \quad =  [A_{st}^{P,1} +A_{st}^{P,2}] v_s+   A_{Tt}^{P,1}\delta u_{sT}  +   A_{Tt}^{P,2}  \delta u_{sT}   -    A_{Tt}^{P,1} A_{sT}^{P,1}u_s . 
	\end{align*}
	Consequently, 
	\begin{align*}
	v_{st}^{P,\natural}(\phi)&  =   \delta v_{st} (\phi ) +  \int_{s}^{t} \left[ (\nabla  v_{r}, \nabla   \phi) + B_P(v_r)(\phi) \right]\,dr   -  v_{s}([A_{st}^{P,1,*} +A_{st}^{P,2,*}]\phi) \\
	& =   v_s ([A_{st}^{P,1,*} +A_{st}^{P,2,*}]\phi) +  \delta u_{sT}  A_{Tt}^{P,1,*}\phi +  \delta u_{sT}  A_{Tt}^{P,2,*}\phi   -  u_s   A_{Tt}^{P,1,*} A_{sT}^{P,1,*}\phi \\
	& \quad + \tilde{u}_{0(t-T)}^{P,\natural}(\phi) \dela{+  u_{s}([A_{sT}^{P,1,*} +A_{sT}^{P,2,*}]\phi)}  + u_{sT}^{P,\natural}(\phi) -  v_{s}([A_{st}^{P,1,*} +A_{st}^{P,2,*}]\phi) \\
	& =  \delta u_{sT} A_{Tt}^{P,2,*}  \phi   + A_{Tt}^{P,1,*}  [ \delta u_{sT}    -  u_s    A_{sT}^{P,1,*} ] \phi  + \tilde{u}_{0(t-T)}^{P,\natural}(\phi) + u_{sT}^{P,\natural}(\phi). 
	\end{align*}
	With $\alpha = \frac{1}{p}$, estimate \eqref{ineq:UBRcontrolestimates} and Lemmata \ref{Thm2.5}-\ref{AprioriVariation1} we obtain
	\begin{align*}
	\Vert v_{st}^{P,\natural} \Vert_{\bH^{-3}} & \le \Vert  \delta u_{sT} \Vert_{\bH^{-1}}  \Vert A_{Tt}^{P,2,*}  \Vert_{\mathcal{L}(\bH^{-1},\bH^{-3})}  + \Vert  \delta u_{sT}    -  u_s    A_{sT}^{P,1,*}  \Vert_{\bH^{-2}}  \Vert A_{Tt}^{P,1,*} \Vert_{\mathcal{L}(\bH^{-2},\bH^{-3})}  \nonumber\\
	& \quad + (\omega_{\tilde{u},\natural}(0,t-T))^{\frac{3}{p}}  + (\omega_{u,\natural}(s,T))^{\frac{3}{p}} \nonumber\\
	& \le (\omega_{A}(s,t))^{\frac{2}{p}}  (\omega_{u}(s,t))^{\frac{1}{p}} + (\omega_{\natural}(s,t))^{\frac{2}{p}}  (\omega_{A}(s,t))^{\frac{1}{p}} + (\tilde{\omega}_{\natural}(s,t))^{\frac{3}{p}}  + (\omega_{u,\natural}(s,t))^{\frac{3}{p}} \\
	& \leq (\omega_{A,u}(s,t))^{\frac{3}{p}}   +  (\omega_{A,\natural}(s,t))^{\frac{3}{p}}  + (\tilde{\omega}_{\natural}(s,t))^{\frac{3}{p}}  + (\omega_{u,\natural}(s,t))^{\frac{3}{p}},
	\end{align*}
	where we have used \cite[Exercise 1.9 part (iii)]{FV_2010} to conclude that $\omega_{A,u} := (\omega_{A}(s,t))^{\frac{2}{3}}  (\omega_{u}(s,t))^{\frac{1}{3}}$ and $\omega_{A,\natural}:= (\omega_{\natural}(s,t))^{\frac{2}{3}}  (\omega_{A}(s,t))^{\frac{1}{3}} $ are controls.

	Hence, if we set $\omega_{P,A,u,\natural} :=  \omega_{A,u} + \omega_{A,\natural} + \tilde{\omega}_{\natural} + \omega_{u,\natural}  $, then we are done with the proof of Lemma~\ref{AL3}. 
\end{proof}

\subsection{General ansatz}\label{sec:ansatz}

Let us fix a rough path $\bZ$. In summary, so far we have shown the existence of a set--valued mapping
\begin{equation}\label{multivalueMapU}
\bD\times\mathcal{C}^{\alpha}_{g,\rm loc}(\T;\bR^K)\to  2^{\bX}, \quad [u_0,  E_0, \bZ] \mapsto \clU[u_0, E_0, \bZ],
\end{equation}
which enjoys the following properties:

\medskip

\begin{enumerate}[label={\bf (A\arabic{*})}]
	
	\item\label{A1} \textbf{Compactness:} For any $[u_0,  E_0, \bZ] \in \bD\times\mathcal{C}^{\alpha}_{g,\rm loc}(\T;\bR^K)$, the set $\clU[u_0,  E_0, \bZ]$ is a non--empty compact
	subset of $\bX$. Indeed, the compactness is equivalent to the weak sequential stability of the solution set which we get from  Theorem \ref{thm-SeqStabilitywrtZandD}. Non-emptiness of $\clU[u_0,  E_0, \bZ]$ follows from Theorem \ref{thm_existence}. 
	
	\item\label{A2} \textbf{Measurability:} The mapping \eqref{multivalueMapU}
	is Borel measurable,  where the range of $\clU$ is endowed with the Hausdorff metric.
Indeed, since $\clU[u_0,  E_0,\bZ]$ is a compact subset of the separable metric space $\bX$,  the Borel measurability of $\clU$  is equivalent to the measurability with respect to the Hausdorff metric on the subspace of compact sets in $2^{\bX}$. Whence,   it is sufficient to apply the following Stroock and Varadhan Lemma  with $Y = \bD$ and $X = \bX$.
	
	\begin{Lemma}\cite[Lemma 12.1.8]{StrVar}\label{lem-strook+varadhan}
		
		Let $Y$ be a metric space and $\clB$ its Borel $\sigma$-field. Let $y \mapsto  K_y$ be a map of $Y$ into $Comp(X)$ for some separable metric space $X$, with $Comp(X)$ the set of all the
		compact subsets of $X$. Suppose for any sequence $y_n \mapsto y$ and $x_n \in K_{y_n}$, it is true that $x_n$ has a limit point $x$ in $K_y$. Then the map $y \mapsto K_y$ is a Borel map of $Y$ into $Comp(X)$.
	\end{Lemma}

	\item\label{A3} \textbf{Shift invariance:} For any $[u, E] \in \clU[u_0,  E_0,\bZ],$
	we have
	\[
	S_T \circ [u , E] \in \clU[u(T),  E(T-), \tilde{\bZ}_{T}]\ \mbox{for any}\ T > 0,
	\]
	where $\tilde{\bZ}_T(t) := S_T \circ \bZ(t)$ for all $t \geq 0$.
	\item\label{A4} \textbf{Continuation:} If $T > 0$, and $[u, E] \in \clU [u_0,  E_0,\bZ],
	\ [\tilde{u},\tilde{E}] \in \clU [u(T), E(T-),\tilde{\bZ}_T],$
	then
	\[
	[u, E] \cup_T [\tilde{u},\tilde{E}]  \in \clU [u_0, E_0, \bZ].
	\]
\end{enumerate}

\subsection{Selection sequence}\label{sec:Selection_seq}
Notice that, the idea for the construction of the selection is to make the set $\clU[u_0, E_0,\bZ]$  smaller and smaller by choosing the arguments of minima of particular functionals. More precisely, following the  arguments presented in \cite{Basaric_2020,Breit_etal_2020, Cardona+Kapitanski_2020}, we consider the following family of Krylov  functionals, see \cite{KrylNV}, 
\begin{equation*}
I_{\lambda,F}[u,E]= \int_{0}^{\infty} e^{-\lambda t}F(u(t),E(t)) dt, \quad \lambda>0,
\end{equation*}
where $F: \bH^{-1} \times \bR \rightarrow \bR$ is a bounded and continuous functional.

Given functional $I_{\lambda,F}$ and a set-valued mapping $\clU$, we define a selection mapping $I_{\lambda,F}\circ \clU$ by
\begin{align}\label{defn-selectionMapping}
& I_{\lambda,F}\circ \clU[u_0, E_0,\bZ] \nonumber \\
&\quad = \{ [u,E] \in \clU[u_0, E_0,\bZ] \ | \ I_{\lambda,F}[u, E] \leq I_{\lambda,F}[\tilde{u}, \tilde{E}] \ \mbox{for all }  [\tilde{u}, \tilde{E}] \in \clU[u_0, E_0,\bZ] \}.
\end{align}
In other words, the selection is choosing arguments of minima of the functional $I_{\lambda,F}$. Observe that, since $I_{\lambda,F}$ is continuous on $\bX$ and the set $\clU[u_0, E_0,\bZ]$ is compact in $\bX$, set $I_{\lambda,F}\circ \clU[u_0, E_0,\bZ] $ is non-empty. 
Our next result says that the set $I_{\lambda, F} \circ \clU$ enjoys the general ansatz if $\clU$ does. Recall that the perturbation rough path $\bZ$ is fixed. 
\begin{Proposition} \label{AP1}
	
	Let $\lambda > 0$ and $F$ be a bounded continuous functional on ${\rm{\bf H}}^{-1} \times \bR$.
	Let the multivalued mapping \eqref{multivalueMapU}
	have the properties \ref{A1}--\ref{A4}.
	Then the map $I_{\lambda,F} \circ \clU$ enjoys \ref{A1}--\ref{A4} as well.	
\end{Proposition}
\begin{proof}[Proof of Proposition \ref{AP1}]
Since the analysis here is pathwise, we observe that compared to the proof of \cite[Proposition 5.1]{Breit_etal_2020} the existence of  $\bZ$ in the system does not create any extra difficulty. Consequently,  the proof follows step by step the lines of  \cite[Proposition 5.1]{Breit_etal_2020}.
 Let us only spell out the proof of the measurability \ref{A2}, since this is of great importance for the measurability of the random dynamical system in Section \ref{sec:rds}.
	\begin{itemize}	
		\item[(\textbf{A2})]
		Let $\mathcal{K} \subset 2^{\bX}$ be  the subspace  of  all the compact subsets of $\bX$.  Note that the map
	\begin{equation}\label{map2}
		\bD\times\mathcal{C}^{\alpha}_{g,\rm loc}(\T;\bR^K)\to  2^{\bX}, \quad [u_0,  E_0, \bZ] \mapsto I_{\lambda, F} \circ \clU[u_0, E_0, \bZ],
		\end{equation}
		takes values in $\mathcal{K}$. Moreover, it is the composition of the Borel measurable map \eqref{multivalueMapU} and the map defined using the continuous functional $\mathcal{I}_{\lambda, F}$
		\begin{equation}\label{map3}
		\mathcal{K}\to \mathcal{K},\quad   K \mapsto \mathcal{I}_{\lambda, F}[K] :=  \argmin_{ K} I_{\lambda, F} .
		\end{equation}
		Hence, \eqref{map2}  is Borel measurable since  \eqref{map3} is Borel measurable due to \cite[Lemma 12.1.7]{StrVar}. 
	\end{itemize}
\end{proof}

Next, we consider the functional $I_{1, \beta}$ where %
\[
\beta(u,  E) = \beta (E), \ \beta: \bR \to \bR \ \mbox{smooth, bounded, and strictly increasing.}
\]
We also recall the following characterization of minimality w.r.t $\prec$, introduced in Definition \ref{defn-admissible}.
\begin{Lemma}\cite[Lemma 5.2]{Breit_etal_2020} \label{CL2}
	Suppose that $[u, E] \in \clU[u_0, E_0, \bZ]$ satisfies
	\[
	\int_0^\infty \exp(-t) \beta( E(t) ) \, dt  \leq
	\int_0^\infty \exp(-t) \beta( \tilde{E}(t) ) \, dt
	\]
	for any $[\tilde{u}, \tilde{E}] \in \clU[u_0, E_0,\bZ]$.
	Then $[u,  E]$ is $\prec$ minimal, meaning, admissible.
\end{Lemma}

Finally, we have all in hand to present the first main result of the present paper.

\begin{Theorem} \label{thm-semiflowSel}
	The Navier-Stokes equation \eqref{NSDiffFormSystem_u} admits a semiflow selection $U$ in the class of weak solutions in the sense of Definition \ref{defn-semiflowSelection}. Moreover, we have that $U\{ u_0, E_0, \bZ \}$ is admissible in the sense of Definition \ref{defn-admissible}, for any $[u_0, E_0, \bZ]\in \bD\times\mathcal{C}^{\alpha}_{g,\rm loc}(\T;\bR^K)$.
\end{Theorem}
\begin{proof} [\textbf{Proof of Theorem \ref{thm-semiflowSel}}]
	First note that by \eqref{defn-selectionMapping} it is clear that the new selection $I_{1,\beta} \circ \clU$ from $\clU$  contains only admissible solutions for any $[u_0, E_0, \bZ]\in \bD\times\mathcal{C}^{\alpha}_{g,\rm loc}(\T;\bR^K)$.
	
	 Next, we choose a countable basis  $\{ \textbf{e}_n \}_{n\in \bN}$ in $\bL^2$, and a countable set $\{ \lambda_k \}_{k\in \bN}$ which is dense in $(0,\infty)$. We consider a countable family of functionals, 
	\begin{align*}
	I_{k,0}[u,  E] &= \int_{0}^{\infty} e^{-\lambda_kt} \beta(E(t)) dt, \\
	I_{k,n}[u,  E]  &= \int_{0}^{\infty} e^{-\lambda_kt} \beta\left( \int_{\bT^3} u(t,\cdot) \cdot \textbf{e}_n dx \right) dt.
	\end{align*}
	The functionals are well defined since $ u(t,\cdot)\in \bH^{-1}(\bT^3; \bR^3)$ for all $t$. Let $\{ (k(j),n(j)) \}_{j=1}^{\infty}$ be an enumeration of the countable set
	\begin{equation*}
	(\bN\times \{ 0 \}) \cup (\bN \times \bN).
	\end{equation*}
	We define
	\begin{equation*}
	\clU^j := I_{k(j),n(j)} \circ \dots \circ I_{k(1),n(1)} \circ I_{1,\beta} \circ \clU, \quad j=1,2, \dots,
	\end{equation*}
	and 
	\begin{equation*}
	\clU^{\infty} := \bigcap_{j=1}^{\infty} \clU^j.
	\end{equation*}
	Next, we  claim that the set-valued mapping
	\begin{equation}\label{defn-Uinfty}
		\bD\times\mathcal{C}^{\alpha}_{g,\rm loc}(\T;\bR^K)\to 2^{\bX},\quad [u_0,  E_0, \bZ] \mapsto \clU^{\infty} [u_0,   E_0,\bZ] ,
	\end{equation}
	enjoys the properties (\textbf{A1})--(\textbf{A4}). Indeed:
	\begin{enumerate}
		\item[(\textbf{A1})] Let us take $[u_0,E_0, \bZ] \in \bD\times\mathcal{C}^{\alpha}_{g,\rm loc}(\T;\bR^K)$. Recall that, from Proposition \ref{AP1}, the set $I_{1, \beta} \circ \clU[u_0,  E_0]$ is compact. Since the sets $\clU^j[u_0, E_0,\bZ]$ are nested:
		\begin{equation*}
		I_{1,\beta} \circ \clU[u_0,  E_0] \supseteq \clU^1[u_0,  E_0] \supseteq \dots \supseteq \clU^j[u_0,  E_0] \supseteq \dots.
		\end{equation*}
		by iterating the procedure of Proposition \ref{AP1}, we get that, for each $j \in \bN$, $\clU^j[u_0,  E_0,\bZ]$ is compact. 
		
		Since $\bX$ is a Hausdorff space and $\clU^{\infty}[u_0,  E_0] $ is a closed subset of $I_{1, \beta} \circ \clU[u_0,  E_0]$,  we infer that $\clU^{\infty}[u_0,  E_0]$ is compact. Moreover, by Proposition \ref{AP1} we know that, for each $j \in \bN$, $\clU^j[u_0,  E_0,\bZ]$ is non-empty. Thus, due to the Cantor intersection theorem we have that $\clU^{\infty}[u_0,  E_0,\bZ] \neq \emptyset$;
		
		\item[(\textbf{A2})] Since the intersection of measurable set--valued maps is measurable, the map \eqref{defn-Uinfty} is measurable. 
		
		\item[(\textbf{A3})] In order to prove the shift invariance property, let $[u_0,E_0, \bZ] \in \bD\times\mathcal{C}^{\alpha}_{g,\rm loc}(\T;\bR^K)$ and $[u,  E]\in \clU^{\infty}[u_0,  E_0,\bZ]$. Thus,  $[u,  E]\in \clU^j[u_0,  E_0,\bZ]$ for every $j \in \bN$. Due to Proposition \ref{AP1}, we know that $I_{1, \beta} \circ \clU$ satisfies the shift invariance property, that is,  if $[u,E] \in I_{1, \beta} \circ \clU[u_0,  E_0,\bZ]$, then 
		$$S_T \circ [u,  E] \in I_{1, \beta} \circ \clU[u(T),  E(T-),\tilde{\bZ}], \mbox{ for all }T>0.$$
		By iterating this procedure we obtain that the shift invariance property holds for every $\clU^j$. This means that for 
		$$ [u,E] \in  \clU^j[u_0,  E_0,\bZ] = I_{k(j),n(j)} \circ \dots \circ I_{k(1),n(1)} \circ I_{1,\beta} \circ \clU[u_0,E_0],$$
		we have
		\begin{equation*}
		S_T \circ [u,  E] \in \clU^j[u(T),  E(T-), \tilde{\bZ}], \mbox{ for all }j \mbox{ and all }T>0.
		\end{equation*}
		Thus
		\begin{equation*}
		S_T \circ [u,  E] \in \clU^{\infty}[u(T), E(T-), \tilde{\bZ}], \mbox{ for all }T>0;
		\end{equation*}
		
		\item[(\textbf{A4})] In order to prove the continuation property, let $T>0$, $[u,E]\in  \clU^{\infty}[u_0,E_0,\bZ]$ and  $$[\tilde{u},\tilde{E}]\in \clU^{\infty}[u(T), E(T-),\tilde{\bZ}].$$
		Then, we have 
		$$[u,E]\in  \clU^j[u_0,  E_0,\bZ], \textrm{ and } [\tilde{u},\tilde{E}]\in \clU^j[u(T),E(T-),\tilde{\bZ}],  \qquad j \in \bN. $$
		By  Proposition \ref{AP1}, we have that $I_{1, \beta} \circ \clU$ satisfies the continuation property, and iterating this procedure we obtain that this property holds for every $\clU^j$. This means that
		\begin{equation*}
		[u, E] \cup_T [\tilde{u}, \tilde{E}] \in \clU^j[u_0,  E_0,\bZ] \mbox{ for all }j \mbox{ and all }T>0.
		\end{equation*}
		Thus
		\begin{equation*}
		[u, E] \cup_T [\tilde{u}, \tilde{E}] \in \clU^{\infty}[u_0,  E_0,\bZ] \mbox{ for all }T>0.
		\end{equation*}
	\end{enumerate}
	Next, we claim that for every $[u_0,E_0, \bZ] \in \bD\times\mathcal{C}^{\alpha}_{g,\rm loc}(\T;\bR^K)$ the set $\clU^{\infty}$ is a singleton, i.e., there exists $U\{ u_0,  E_0,\bZ \} \in \bX$ such that
	\begin{equation}\label{Umap}
	\clU^{\infty}[u_0,  E_0,\bZ]= \big\{ U\{ u_0,  E_0,\bZ \} \big\}. 
	\end{equation}
	To prove this, first observe that by \eqref{defn-selectionMapping}, for any $[u^1,  E^1], [u^2,  E^2]\in \clU^{\infty}[u_0,  E_0,\bZ]$, 
	\begin{equation*}
		I_{k(j),n(j)} [u^1, E^1] = I_{k(j),n(j)} [u^2, E^2], \qquad j \in \bN.
	\end{equation*}
	 Since  the integrals $I_{k(j), n(j)}$ can be seen as Laplace transforms 
	\begin{equation*}
	F(\lambda_k)= \int_{0}^{\infty} e^{-\lambda_k t}f(t)dt,
	\end{equation*}
	of the functions
	\begin{equation*}
	f\in \left\{ \beta(E), \  \beta\left(\int_{\bT^3}  u\cdot \textbf{e}_n dx \right)\right\}, 
	\end{equation*}
	the Lerch theorem \cite[Theorem 2.1]{Cohen}  implies that 
	\begin{align*}
	\beta(E^1(t)) &= \beta(E^2(t)), \\
	\beta\left(\int_{\bT^3} u^1(t,\cdot)\cdot \textbf{e}_n dx \right) &= \beta\left(\int_{\bT^3} u^2(t,\cdot)\cdot \textbf{e}_n dx \right),
	\end{align*}
	for all $n \in \bN$ and for a.e. $t\in (0,\infty)$. As $\beta$ is strictly increasing, we must have 
	\begin{equation*}
	E^1(t-)= E^2(t-), \quad \langle u^1(t,\cdot), \textbf{e}_n \rangle_{\bL^2}= \langle u^2(t,\cdot), \textbf{e}_n \rangle_{\bL^2},
	\end{equation*}
	for all $n \in \bN$ and for a.e. $t\in (0,\infty)$. Since $\{ \textbf{e}_n \}_{n\in \bN}$ form a basis in $\bL^2$, we deduce that
	\begin{equation*}
	u^1=u^2, \mbox{ and } E^1=E^2 \mbox{ a.e. on }(0,\infty).
	\end{equation*}
	Due to \eqref{Umap}, measurability of $U$ follows from (\textbf{A2}) for $\clU^\infty$. While  the semigroup property follows from (\textbf{A3}). Indeed, for $t_1,t_2 \geq 0$ it holds 
	\begin{equation*}
	U\{ u_0,  E_0,\bZ \} (t_1+t_2)= S_{t_1} \circ U\{ u_0,  E_0,\bZ \} (t_2) = U\{ U\{u_{0},E_{0},\bZ\}(t_{1}), \tilde{\bZ}_{t_{1}} \} (t_2),
	\end{equation*}
	where $\tilde{\bZ}_{t_1}(t_2) := \bZ (t_1+t_2)$. 
	This completes the proof of Theorem \ref{thm-semiflowSel}.
\end{proof}

\begin{Remark}\label{r3.10}
	It is important to highlight that one can introduce a new selection, associated with the considered Navier-Stokes equation  \eqref{NSDiffFormSystem_u}, defined only in terms of the initial velocity. However, in this case we can only achieve that, for each rough  path $\bZ$,  the semigroup property holds almost everywhere in time. The proof of this argument in our framework is similar to \cite[Section 5]{Basaric_2020}, where the author proves this claim for the compressible Navier-Stokes system without any perturbation. 
\end{Remark}

\section{Random dynamical system}\label{sec:rds}

Based on the semiflow selection from  the previous section  we investigate the existence of a  random dynamical system for Navier-Stokes equation \eqref{NSDiffFormSystem_u}.

Let $(\Omega, \mathcal{F})$
be a measurable space. A family $\theta =
(\theta_t)_{t\in\mathbb{T}}$ of maps from $\Omega$ to itself is called a \textbf{measurable
dynamical system} provided
\begin{enumerate}
	\item $(t, \omega) \mapsto \theta_t \omega$ is $\clB (\T)
	\otimes \mathcal{F} / \mathcal{F}$-measurable, where $\clB(\T)$ is the Borel sigma of $\T$,
	
	\item $\theta_0 = \tmop{Id}_{\Omega}$,
	
	\item $\theta_{s + t} = \theta_t \circ \theta_s$ for all $s, t \in \T$.
\end{enumerate}
If $\mathbb{P}$ is a probability measure on $(\Omega, \mathcal{F})$ that is invariant under $\theta$, i.e. $\mathbb{P} \circ \theta^{- 1}_t =\mathbb{P}$ for all $t \in \T$, we call  the quadruple $(\Omega, \mathcal{F}, \mathbb{P}, \theta)$ a measurable metric dynamical system.

The following is taken from L. Arnold's book, see \cite[Definition 1.1.1]{Arnold_1998B}. 

\begin{Definition}\label{def:rds}
	A \textbf{measurable  random dynamical system} (MRDS) on a measurable space $(X, \clX)$, over a metric dynamical system $(\Omega, \mathcal{F}, \mathbb{P}, (\theta_{t})_{t \in {\T}})$ with time ${\T}$ is a mapping $$\Phi : \mathbb{T} \times \Omega \times X \to X, \quad (t, \omega, x) \mapsto \Phi(t, \omega, x)\,$$ with the following properties
	\begin{enumerate}[i)]
		\item \textit{Measurability}: $\Phi$ is $(\clB(\mathbb{T}) \otimes \mathcal{F} \otimes \clX)/\clX$ measurable.
		\item \textit{Cocycle property}: The mappings $\Phi(t, \omega) \defeq \Phi(t, \omega, \cdot) : X \to X$ form a cocycle over $\theta$, i.e. they satisfy
		\begin{subequations}
			\begin{equation}
			\label{eq:1}
			\Phi(0, \omega) = \id_X \quad \forall \omega \in \Omega, 
			\end{equation}
			\begin{equation}
			\label{eq:2}
			\Phi(t+s, \omega) = \Phi(t,\theta_{s}\omega) \smallcirc \Phi(s, \omega) \quad \forall s,t \in \mathbb{T}, \omega \in \Omega.
			\end{equation}
		\end{subequations}
	\end{enumerate}
\end{Definition}

\begin{Remark}
	If the mapping $\Phi$ in Definition~\ref{def:rds} does not depend on $\omega$, then the dynamics on $X$ is independent \dela{of that }of the underlying dynamical system on $\Omega$, and $\Phi(t)$ satisfies the semigroup property. 
\end{Remark}

Let us fix  a measurable metric dynamical system $(\Omega, \mathcal{F}, \mathbb{P}, (\theta_{t})_{t \in\T})$. As defined in \cite[Section 2]{BRS_2017}, for $\alpha \in \left( \frac{1}{3}, \frac{1}{2} \right] $, we say that a measurable map 
\[ \bZ = (Z, \mathbb{Z}) : \Omega \rightarrow C_{2,\rm loc}^{\alpha} (\T; \bR^K) \times C_{2,\rm loc}^{2 \alpha} (\T ; \bR^{K
	\times K}), \]
is a geometric $\alpha$-H\"older rough path cocycle provided $\bZ (\omega)$ is a geometric $\alpha$-H\"older rough path and the following
cocycle property is satisfied
\[ Z_{s, s + t} (\omega) = Z_{0,t} (\theta_s \omega), \qquad \mathbb{Z}_{s, s + t}
(\omega) =\mathbb{Z}_{0, t} (\theta_s \omega), \]
holds true for every $s, t \in \T$ and $\omega \in \Omega$.

For the sake of completeness we include the following simple observation regarding the shift-property of an $\alpha$-H\"older rough path.

\begin{Lemma}\label{lem-rpShift}
Let $\bZ = (Z, \Z)$ be a geometric $\alpha$-H\"older rough path cocycle for some $\alpha \in \left( \frac{1}{3}, \frac{1}{2} \right] $. 	For every $0 \leq s \leq t$, $h >0$ and $\omega \in \Omega$, we have $Z_{s+h,t+h}(\omega) = Z_{s,t}(\theta_h \omega)$ and $\Z_{s+h,t+h}(\omega) = \Z_{s,t}(\theta_h \omega)$.

\end{Lemma}

Finally, we have all in hand to formulate and prove the second main result of the present paper.

\begin{Theorem}\label{thm-rds}
	Assume that, for given measurable metric dynamical system $(\Omega, \mathcal{F}, \mathbb{P}, \theta)$,  the driving rough path $\bZ = (Z,\Z)$ is a geometric $\alpha$-H\"older rough path cocycle for some $\alpha \in \left( \frac{1}{3}, \frac{1}{2}\right]$.  Then the Navier-Stokes system \eqref{NSDiffFormSystem_u} generates a measurable random dynamical system on $\bD$.
\end{Theorem}
\begin{proof}[Proof of Theorem \ref{thm-rds}]
Given the random rough path $\bZ$
	%
	 and the semiflow $U$ constructed in Section \ref{sec:semiflow} we define 
	$$\varphi:  \Omega \times \bD \to C_{\loc}\bH_w^0  \times L_{\loc}^1(\T),\quad  (\omega,[u_0,E_0]) \mapsto U\{ u_0,E_0,\bZ(\omega)\}   .$$
By the definition of a rough path cocycle, this map factorizes as
$$
(\omega,[u_{0}, E_{0}])\mapsto (u_{0}, E_{0}, \bZ(\omega))\mapsto U\{u_{0}, E_{0}, \bZ(\omega)\}
$$
hence it is well-defined and measurable due to Theorem~\ref{thm-semiflowSel}. This is the point where the Wong-Zakai stability in the rough path setting becomes essential. Notice that, since $\varphi(\omega,[u_0,E_0]) \in \clU^\infty[u_0,E_0, \bZ(\omega)]$, we can evaluate it pointwise with respect to $t \in \T$. 

We claim that
	\begin{equation}\label{rdsMap}
		\Phi: \T \times \Omega \times \bD \to \bD,\quad (t,\omega,[u_0,E_0]) \mapsto \varphi(\omega,[u_0,E_0]) (t) ,
	\end{equation}
	is a measurable random dynamical system. To prove the claim, first observe that the map $\Phi$ is well-defined.   Next, it is clear that the measurability of $\Phi$ can be deduced if we show that for given $\omega \in \Omega$ and $[u_0,E_0] \in \bD$, 
	$$\T\to \bD ,\quad t \mapsto \varphi(\omega,[u_0,E_0]) (t)   \quad \textrm{ is measurable}. $$ 
	%
	%
	But this is indeed the case, because if $[u,E]=\varphi(\omega,[u_0,E_0])$ then $\T\to \bH^0 $, $ t \mapsto u_t $, is weakly continuous and the measurability of $\T \to \bR_+$, $ t \mapsto E(t-)  $, is a consequence of the Lebesgue differentiation theorem, since 
	$$E(t-) = \lim\limits_{h \to 0} \frac{1}{h} \int_{h}^{t+h} E(s) \, ds. $$

It remains to verify the cocycle property of $\Phi$. In view of  the definition of $\Phi$, the semiflow property of $U$ as well as Lemma~\ref{lem-rpShift}, we infer for all $t,s\in\mathbb{T}$, $\omega\in\Omega$
$$
\begin{aligned}
\Phi(t+s,\omega)([u_{0},E_{0}])&=\varphi_{t+s}(\omega)([u_{0},E_{0}])=U\{u_{0},E_{0},\bZ(\omega)\}(t+s)\\
&=U\{U\{u_{0},E_{0},\bZ(\omega)\}(s),\tilde\bZ_{s}(\omega)\}(t)\\
&=U\{U\{u_{0},E_{0},\bZ(\omega)\}(s),\bZ(\theta_{s}\omega)\}(t)\\
&=\varphi_{t}(\theta_{s}\omega)\circ \varphi_{s}(\omega)([u_{0},E_{0}])=\Phi(t,\theta_{s}\omega)\circ\Phi(s,\omega)([u_{0},E_{0}]),
\end{aligned}
$$
which completes the proof.
\end{proof}

\appendix
\section{A priori estimate and compactness}\label{appx:apriori_compactness}

Here we state, without proof, all the required a priori estimates from  \cite[Section 3]{Hofmanova_etal_2019}. Let us fix any $T >0$ and  assume that $u$ is the first component of a weak solution to  \eqref{NSDiffFormSystem_u} according to Definition~\ref{def-weaksolution}.

\begin{Lemma}\cite[Lemma 3.1]{Hofmanova_etal_2019} \label{Thm2.5}
	 For $(s,t)\in \Delta_T$ such that  $\varpi(s,t)\le L$, let $\omega_{P, \natural }(s,t) := |u^{P, \natural}|^{\frac{p}{3}}_{\frac{p}{3} -var; [s,t];\bH^{-3}}.$
	Then  there is a constant $\tilde{L}>0$, depending only on $p$ and $d$,  such that  for all $(s,t)\in \Delta_T$ with  $\varpi(s,t)\le L$ and $\omega_{A}(s,t)\leq \tilde{L}$,
	\begin{equation} \label{RemainderEstimateMu}
	\omega_{P, \natural }(s,t)  \lesssim_{p}|u|^{\frac{p}{3}}_{L^{\infty}_T\bH^0}  \omega_A(s,t) + \omega_{\mu}(s,t)^{\frac{p}{3}} (\omega_A(s,t)^{\frac{1}{3}}  + \omega_A(s,t)^{\frac{2}{3}}),
	\end{equation}
	and
	\begin{equation} \label{RemainderEstimateWithoutMu}
	\omega_{P, \natural }(s,t)  \lesssim_{p} |u|^{\frac{p}{3}}_{L^{\infty}_T\bH^0}  \omega_A(s,t) + ( 1 + |u|_{L^{\infty}_T\bH^0} )^{\frac{2p}{3}}(t-s)^{\frac{p}{3}} \omega_A(s,t)^{\frac{1}{12}}.
	\end{equation}
\end{Lemma}

\begin{Lemma}\cite[Lemma 3.3]{Hofmanova_etal_2019}  \label{AprioriVariation}
	Solution  $u$ belongs to $C^{p-\textnormal{var}}([0,T];\bH^{-1})$ and  there is a constant $\tilde{L}>0$, depending only on $p$ and $d$,   such that for all $(s,t)\in \Delta_T$ with $\varpi(s,t)\le L$, $\omega_{A}(s,t)\leq \tilde{L}$, and $\omega_{P,\natural}(s,t)\leq \tilde{L}$, it holds that
	$$
	\omega_u(s,t) \lesssim_{p} (1 + |u|_{L^{\infty}_T\bH^0})^{p} (\omega_{P, \natural }(s,t) + \omega_{\mu}(s,t)^{p} + \omega_A(s,t)), 
	$$
	where $\omega_u (s,t) := |u|^p_{p - \textnormal{var}; [s,t];\bH^{-1}}$.
\end{Lemma}

\begin{Lemma} \cite[Lemma 3.4]{Hofmanova_etal_2019} \label{AprioriVariation1}
	The remainder $u^{\sharp}$ is in $C_2^{\frac{p}{2}-\textnormal{var}}([0,T];\bH^{-2})$ and  there is a constant $\tilde{L}>0$, depending only on $p$ and $d$,   such that for all $(s,t)\in \Delta_T$ with $\varpi(s,t)\le L$, $\omega_{A}(s,t)\leq \tilde{L}$, and $\omega_{P,\natural}(s,t)\leq \tilde{L}$, it holds that
	$$
	\omega_\sharp(s,t) \lesssim_{p}  (1 + |u|_{L^{\infty}_T\bH^0})^{\frac{p}{2}} (\omega_{P, \natural }(s,t) + \omega_{\mu}(s,t)^\frac{p}{2} + \omega_A(s,t)), 
	$$
	where $\omega_\sharp (s,t) := |u^\sharp|^\frac{p}{2}_{\frac{p}{2} - \textnormal{var}; [s,t];\bH^{-2}}$.
\end{Lemma}

The following compact embedding result is useful in the proof of Sequential Stability Theorem~\ref{thm-SeqStabilitywrtZandD}. 
\begin{Lemma}\cite[Lemma A.2]{Hofmanova_etal_2019} \label{CompactnessLemma}
	Let  $\omega$ and $\varpi$  be a controls on $[0,T]$ and $L,\kappa >0$. Let
	$$
	X=L_T^2\bH^1 \cap \left\{ g  \in C_T\bH^{-1}  : |\delta g_{st}|_{-1} \leq \omega(s,t)^{\kappa}, \;\forall (s, t)  \in \Delta_{T} \textnormal{  with } \varpi(s,t) \leq L   \right\}
	$$
	be endowed with the norm
	$$
	|g|_X=|g|_{L_T^2\bH^1}+\sup_{t\in [0,T]}|g_t|_{-1}+\sup\left\{\frac{ |\delta g_{st}|_{-1}}{\omega(s,t)^{\kappa}}: (s,t)\in \Delta_T \;\textnormal{ s.t.} \;\varpi(s,t) \leq L\right\}.
	$$
	Then $X$ is compactly embedded into $C_T\bH^{-1}$ and $L^2_T \bH^0$.
\end{Lemma}

\def\cprime{$'$} \def\ocirc#1{\ifmmode\setbox0=\hbox{$#1$}\dimen0=\ht0
  \advance\dimen0 by1pt\rlap{\hbox to\wd0{\hss\raise\dimen0
  \hbox{\hskip.2em$\scriptscriptstyle\circ$}\hss}}#1\else {\accent"17 #1}\fi}


\begin{thebibliography}{10}

\bibitem{Arnold_1998B}
L. Arnold.
\newblock {\em Random Dynamical Systems}.
\newblock Springer Monographs in Mathematics. Springer-Verlag, Berlin, 1998. xvi+586 pp. ISBN: 3-540-63758-3.

\bibitem{Bailluel+Gub_2017}
I. Bailleul and  M. Gubinelli.
\newblock {\em Unbounded rough drivers}.
\newblock {\em  Ann. Fac. Sci. Toulouse Math.}, {\bf 6}(26):795--830, 2017.   

\bibitem{BRS_2017}
I. Bailleul, S. Riedel and M. Scheutzow. 
\newblock {\em Random dynamical systems, rough paths and rough flows}.
\newblock {\em  J. Differential Equations}, {\bf 262}(12):5792--5823, 2017.   

\bibitem{Ball_1997}
J. M. Ball.
\newblock {\em Continuity properties and global attractors of generalized semiflows and the Navier-Stokes equations}.
\newblock {\em  J. Nonlinear Sci.}, {\bf 7}(5):475--502, 1997.   


\bibitem{Basaric_2020}
D. Basari\'c. 
\newblock {\em Semiflow selection for the compressible Navier-Stokes system}.
\newblock {\em  J. Evol. Equ.}, {\bf 21}(1):277--295, 2021.   




\bibitem{Breit_etal_2020}
D. Breit, E. Feireisl and M. Hofmanov\'a.
\newblock {\em Solution semiflow to the isentropic Euler system}.
\newblock {\em  Arch. Ration. Mech. Anal.}, {\bf 235}(1):167--194, 2020.   

\bibitem{Breit_etal_2020-a}
D. Breit, E. Feireisl and M. Hofmanov\'a.
\newblock {\em Dissipative solutions and semiflow selection for the complete Euler system}.
\newblock {\em   Comm. Math. Phys.}, {\bf 376}(2):1471--1497, 2020.  


\bibitem{ZB+Li_2006}
Z. Brze\'zniak and Y. Li.
\newblock {\em Asymptotic compactness and absorbing sets for 2D stochastic Navier-Stokes equations on some unbounded domains}.
\newblock {\em  Trans. Amer. Math. Soc.}, {\bf 358}(12):5587--5629, 2006.   


\bibitem{Cardona+Kapitanski_2020}
J. E. Cardona  and  L. Kapitanskii.
\newblock {\em Semiflow selection and Markov selection theorems}.
\newblock {\em  Topol. Methods Nonlinear Anal.}, {\bf 56}(1):197--227, 2020.  


\bibitem{Crauel+Flandoli_1994}
H. Crauel  and F. Flandoli.
\newblock {\em Attractors for random dynamical systems}.
\newblock {\em  Probab. Theory Related Fields}, {\bf 100}(3):365--393, 1994. 

\bibitem{Cohen}
A. M.  Cohen.
\newblock {\em Numerical methods for Laplace transform inversion}.
\newblock  Numerical Methods and Algorithms, 5. Springer, New York, 2007.  

\bibitem{DPD}
G. Da Prato and A. Debussche.
\newblock {\em Ergodicity for the 3D stochastic Navier-Stokes equations}.
\newblock {\em  J. Math. Pures Appl.}, {\bf 82}(8):877--947,2003. 

\bibitem{Daf4}
C. M. Dafermos.
\newblock {\em The second law of thermodynamics and stability}.
\newblock {\em Arch. Rational Mech. Anal.}, {\bf 70}:167--179, 1979.

\bibitem{Deya_etal_2019}
A. Deya, M.  Gubinelli, M. Hofmanov\'a  and  S. Tindel.
\newblock {\em A priori estimates for rough PDEs with application to rough conservation laws}.
\newblock {\em   J. Funct. Anal.}, {\bf 2716}(12):3577--3645, 2019. 

\bibitem{Elworthy_1982B}
K. D. Elworthy.
\newblock {\em Stochastic differential equations on manifolds}.
\newblock London Mathematical Society Lecture Note Series, 70. Cambridge University Press, Cambridge-New York, 1982. 

\bibitem{Flandoli_1995B}
F. Flandoli.
\newblock {\em Regularity theory and stochastic flows for parabolic SPDEs}.
\newblock {\em  J. Evol. Equ.}, {\bf 19}(1):203--247, 2019. 



\bibitem{Flandoli_etal_2020}
F. Flandoli, M. Hofmanov\'a, D. Luo, and T. Nilssen.
\newblock {\em Global well-posedness of the 3D Navier--Stokes equations perturbed by a deterministic vector field}.
\newblock {\em  	arXiv:2004.07528}. 

\bibitem{Flandoli+Romito_2008}
F.~Flandoli and M.~Romito.
\newblock Markov selections for the 3{D} stochastic {N}avier-{S}tokes
equations.
\newblock {\em Probab. Theory Related Fields}, {\bf 140}(3-4):407--458, 2008.


\bibitem{Flandoli+Schmalfuss_1996}
F. Flandoli and B. Schmalfuss.
\newblock {\em Random attractors for the 3D stochastic Navier-Stokes equation with multiplicative white noise}.
\newblock {\em  Stochastics Stochastics Rep.}, {\bf 59}(1-2):21--45, 1996. 
 

\bibitem{FH_2014}
P. K. Friz and M. Hairer.
\newblock {\em A course on rough paths. With an introduction to regularity structures}.
\newblock Universitext. Springer, Cham,  2014. 

\bibitem{FV_2010}
P. K. Friz and N. B. Victoir.
\newblock {\em Multidimensional stochastic processes as rough paths. Theory and applications}.
\newblock Cambridge Studies in Advanced Mathematics, 120. Cambridge University Press, Cambridge, 2010. 



\bibitem{Schmalfuss_etal_2010}
M. J. Garrido-Atienza, K. Lu and B. Schmalfuss.
\newblock {\em Random dynamical systems for stochastic partial differential equations driven by a fractional Brownian motion}.
\newblock {\em  Contin. Dyn. Syst. Ser. B}, {\bf 14}(2):473--493, 2010. 

\bibitem{Schmalfuss_etal_2016}
M. J. Garrido-Atienza, K. Lu and B. Schmalfuss.
\newblock {\em Random dynamical systems for stochastic evolution equations driven by multiplicative fractional Brownian noise with Hurst parameters $H \in (1/3,1/2]$}.
\newblock {\em  SIAM J. Appl. Dyn. Syst.}, {\bf 15}(1):625--654, 2016. 

\bibitem{Gess_etal_2011}
B. Gess, W.  Liu and  M. R\"ockner.
\newblock {\em Random attractors for a class of stochastic partial differential equations driven by general additive noise}.
\newblock {\em  J. Differential Equations}, {\bf 251}(4-5):1225--1253, 2011.   






\bibitem{Hesse+Neamtu_2020}
R. Hesse and A. Neam\c tu. 
\newblock {\em Global solutions and random dynamical systems for rough evolution equations}.
\newblock {\em  Discrete Contin. Dyn. Syst. Ser. B}, {\bf 25}(7):2723--2748, 2020.   


\bibitem{Hofmanova_etal_2019}
M. Hofmanov\'a, J.-M. Leahy  and T. Nilssen. 
\newblock {\em On the Navier-Stokes equation perturbed by rough transport noise}.
\newblock {\em  J. Evol. Equ.}, {\bf 19}(1):203--247, 2019.

\bibitem{Hofmanova_etal_2021+}
M. Hofmanov\'a, J.-M. Leahy  and T. Nilssen. 
\newblock {\em On a rough perturbation of the Navier-Stokes system and its vorticity formulation}.
\newblock {\em  Ann. Appl. Probab.}, {\bf 31}(2):736--777, 2021.   


\bibitem{HZZ20}
M. Hofmanov\'a, R. Zhu, X. Zhu.
\newblock {\em On ill- and well-posedness of dissipative martingale solutions to stochastic 3D Euler equations}.
\newblock {\em  to appear in Comm. Pure Appl. Math.,	arXiv:2009.09552}. 



\bibitem{KrylNV}
N.~V. Krylov.
\newblock The selection of a {M}arkov process from a {M}arkov system of
processes, and the construction of quasidiffusion processes.
\newblock {\em Izv. Akad. Nauk SSSR Ser. Mat.}, {\bf 37}:691--708, 1973.


\bibitem{Kunita_1990B}
H.  Kunita.
\newblock {\em Stochastic flows and stochastic differential equations}.
\newblock Cambridge Studies in Advanced Mathematics, 24. Cambridge University Press, Cambridge, 1990. 

\bibitem{Lyons_1998}
T. J. Lyons. 
\newblock {\em Differential equations driven by rough signals}.
\newblock {\em  	 Rev. Mat. Iberoamericana}, {\bf 14}(2):215--310, 1998. 


\bibitem{Rubio+Robinson_2003}
P. Mar\'in-Rubio and J. C. Robinson. 
\newblock {\em Attractors for the stochastic 3D Navier-Stokes equations}.
\newblock {\em  Stoch. Dyn.}, {\bf 3}(3):279--297, 2003. 




\bibitem{mikulevicius2001note}
R.~Mikulevicius, B.~Rozovskii, et~al.
\newblock A note on {K}rylov's $ l\_p $-theory for systems of {SPDE}s.
\newblock {\em Electronic Journal of Probability}, 6, 2001.

%


\bibitem{mikulevicius2003cauchy}
R. ~Mikulevicius.
\newblock On the {C}auchy problem for the stochastic {S}tokes equations.
\newblock {\em SIAM Journal on Mathematical Analysis}, 34(1):121--141, 2002.


\bibitem{Scheutzow_1996}
M. Scheutzow. 
\newblock {\em On the perfection of crude cocycles}.
\newblock {\em  Random Comput. Dynam.}, {\bf 4}(4):235--255, 1996.

\bibitem{Sell_1996}
G. R. Sell. 
\newblock {\em Global attractors for the three-dimensional Navier-Stokes equations}.
\newblock {\em  J. Dynam. Differential Equations}, {\bf 8}(1):1--33, 1996. 






\bibitem{StrVar}
D.~W. Stroock and S.~R.~S. Varadhan.
\newblock {\em Multidimensional diffusion processes}.
\newblock Classics in Mathematics. Springer-Verlag, Berlin, 2006.
\newblock Reprint of the 1997 edition.



\bibitem{RT83}
R.~Temam.
\newblock {\em {N}avier-{S}tokes equations and nonlinear functional analysis},
volume~41 of {\em CBMS-NSF Regional Conference Series in Applied
	Mathematics}.
\newblock Society for Industrial and Applied Mathematics (SIAM), Philadelphia,
PA, 1983.







\end{thebibliography}
\end{document}